\documentclass[11pt]{amsart}

\usepackage{epigamath-voisin}

\usepackage[english]{babel}

\numberwithin{equation}{section}

\usepackage[shortlabels]{enumitem}
\setlist[enumerate,1]{label={\rm(\arabic*)}, ref={\rm\arabic*}} 

\usepackage{amsmath,amsfonts,amssymb,amsthm,amscd}
\usepackage[utf8]{inputenc}
\usepackage[english]{babel}

\usepackage{tabto}
\usepackage{standalone}
\usepackage{tikz}
\usetikzlibrary{matrix}
\usetikzlibrary{arrows}
\usepackage{xfrac} 
\usepackage[all]{xy} 

\usepackage{xcolor}
\usepackage{aliascnt}

\theoremstyle{plain}
\newtheorem{thm}{Theorem}[section]

\newaliascnt{lem}{thm}
\newtheorem{lem}[lem]{Lemma}
\aliascntresetthe{lem}

\newaliascnt{cor}{thm}
\newtheorem{cor}[cor]{Corollary}
\aliascntresetthe{cor}

\newaliascnt{prop}{thm}
\newtheorem{prop}[prop]{Proposition}
\aliascntresetthe{prop}

\newaliascnt{quest}{thm}

\aliascntresetthe{quest}

\theoremstyle{definition}

\newaliascnt{defn}{thm}

\aliascntresetthe{defn}

\theoremstyle{remark}

\newaliascnt{rem}{thm}
\newtheorem{rem}[rem]{Remark}
\aliascntresetthe{rem}

\newaliascnt{ex}{thm}

\aliascntresetthe{ex}

\def\cA{\ensuremath{\mathcal{A}}}

\def\cC{\ensuremath{\mathcal{C}}}

\def\cM{\ensuremath{\mathcal{M}}}

\def\cO{\ensuremath{\mathcal{O}}}

\def\cR{\ensuremath{\mathcal{R}}}

\newcommand{\lra}{\longrightarrow}

\DeclareMathOperator{\Pic}{Pic}

\DeclareMathOperator{\Sym}{Sym}
\DeclareMathOperator{\Ker}{Ker}
\DeclareMathOperator{\Id}{Id}
\DeclareMathOperator{\Hom}{Hom}
\renewcommand{\div}{\operatorname{div}}
\renewcommand{\Im}{\operatorname{Im}}
\newcommand{\Supp}{\operatorname{Supp}}
\definecolor{applegreen}{rgb}{0.55, 0.71, 0.0}
\makeatletter
\newcommand{\mylabel}[2]{#2\def\@currentlabel{#2}\label{#1}}
\makeatother

\newcommand\restr[2]{{% we make the whole thing an ordinary symbol
 \left.\kern-\nulldelimiterspace % automatically resize the bar with \right
 #1 % the function
 \vphantom{\big|} % pretend it's a little taller at normal size
 \right|_{#2} % this is the delimiter
}}

\YearArticle{2023} \EpigaArticleNr{9} \ReceivedOn{August 24, 2022}
\InFinalFormOn{June 29, 2023}
\AcceptedOn{July 31, 2023}

\title{The second fundamental form on the moduli space \\of cubic threefolds in $\boldsymbol{\mathcal A_5}$}
\titlemark{2nd fundamental form on cubic 3folds}

\author{Elisabetta Colombo} 
\address{Universit\`a degli Studi di Pavia, Dipartimento di Matematica, Via Cesare Saldini 50, 20133 Milano, Italy }
\email{elisabetta.colombo@unimi.it}
\author{Paola Frediani}
\address{Universit\`a degli Studi di Pavia, Dipartimento di Matematica, Via Ferrata 1, 27100 Pavia, Italy  }
\email{paola.frediani@unipv.it}
\author{Juan Carlos Naranjo}
\address{Departament de Matem\`atiques i Inform\`atica,
Universitat de Barcelona, Gran Via de les Corts Catalanes, 585, 08007 Barcelona, Spain \newline Centre de Recerca Matem\`atica, Edifici C, Campus Bellaterra, 08193 Bellaterra, Spain}
\email{jcnaranjo@ub.edu}
\author{Gian Pietro Pirola}
 \address{Universit\`a degli Studi di Pavia, Dipartimento di Matematica, Via Ferrata 1, 27100 Pavia, Italy  }
 \email{gianpietro.pirola@unipv.it}

\authormark{E.~Colombo, P.~Frediani, J.\,C.~Naranjo, G.\,P.~Pirola}

\AbstractInEnglish{We study the second fundamental form of the Siegel metric in $\mathcal A_5$ restricted to the locus of intermediate Jacobians of cubic threefolds. We prove that the image of this second fundamental form, which is known to be non-trivial, is contained in the kernel of a suitable multiplication map. Some ingredients are: the conic  bundle structure of cubic threefolds, Prym theory, Gaussian maps and Jacobian ideals.}


\acknowledgement{E.~Colombo, P.~Frediani and G.\,P.~Pirola are members of Gnsaga (INDAM) and are partially supported by PRIN project Moduli spaces and Lie theory (2017). P.~Frediani and G.\,P.~Pirola are partially supported by MIUR: Dipartimenti di Eccellenza Program (2018-2022) - Dept.\ of Math.\ Univ.\ of Pavia. J.\,C.~Naranjo is partially supported by the Spanish MINECO research project PID2019-104047GB-I00.}

\dedication{Dedicated to Claire Voisin with great admiration}

\begin{document}

%%%%%%%%%%%%%%%%%%%%%%%%%%%%%%%
% Title page
%%%%%%%%%%%%%%%%%%%%%%%%%%%%%%%

\maketitle

\begin{prelims}

\DisplayAbstractInEnglish

\bigskip

\DisplayKeyWords

\medskip

\DisplayMSCclass

\end{prelims}

%%%%%%%%%%%%%%%%%%%%%
% Table of Contents
%%%%%%%%%%%%%%%%%%%%%

\newpage

\setcounter{tocdepth}{1}

\tableofcontents

%%%%%%%%%%%%%%%%%%%%%
% Content begins here
%%%%%%%%%%%%%%%%%%%%%

\section{Introduction}\label{sec1}

This paper deals with the local geometry of the period map of the moduli space of cubic threefolds. This is the map that associates with
a cubic its intermediate Jacobian, which turns out to be a principally polarized abelian variety (ppav from now on) of dimension $5$.   We are interested in the infinitesimal behavior of the intermediate Jacobian, which is a fundamental topic with a long tradition. In the seminal paper \cite{cg}, it was proved that the period map is an orbifold embedding and the image ${\mathcal C}$ does not intersect the Jacobian locus. In particular, global and infinitesimal Torelli theorems hold (see also \cite{carl_g} and \cite{tj}). Moreover, the deformations of the cubic threefold correspond to deformations of its Fano surface. The infinitesimal variation of the normal function attached to the Fano surface has been studied in \cite{cnp}.

Our aim here is to study the locus  $\mathcal C$ of intermediate Jacobians from a different point of view: $\mathcal C$  inherits an orbifold natural K\"ahler form $h$ from the moduli space  $\mathcal A_5$ of ppavs of dimension $5$, which is the one induced by the natural symmetric form on the Siegel space (see \cite{sat}). Following the case of the moduli space of curves,  we study the second fundamental form associated via this embedding with the Siegel form $h$.

This map is a non-holomorphic object due to the fact that the Hodge decomposition is involved in its definition, but it contains basic information about the period map. We observe that the local geometry of the image of the period map is far from being understood. It is worthwhile to recall the case of the moduli of curves, that was the model for the present research.  An important piece of information about that comes from the second Gaussian map, and there is a complete description as a multiplication map. The first part has been pursued in \cite{cpt} using a theory where Gaussian maps and Hodge theory are mixed up, while the multiplication description was provided in \cite{cfg} some years later.

In this paper we are able to develop the first part of this
program. We will use that $\mathcal C$  is contained in the  Prym locus; more precisely, each intermediate Jacobian is the Prym variety of an odd double covering of a smooth  plane quintic. Then we consider the Prym map and its relation with the Hodge-Gaussian maps to investigate the second fundamental form on $\mathcal C$.

We prove a strong unexpected symmetry on our second fundamental form that we now describe. 

We have the cotangent exact sequence
\[0 \longrightarrow {N_{\cC\slash \cA_5}^*} \longrightarrow {\Omega^1_{\cA_5}}|_{\cC} \stackrel{q}\longrightarrow \Omega^1_{\cC} \longrightarrow 0.
\]

Denote by $\nabla$ the Chern connection of the Siegel metric, and consider the second fundamental form 
\[
II_{\cC \slash \cA_5}=(q \otimes \Id) \circ \nabla|_{ {N_{\cC\slash \cA_5}^*}}\colon   {N_{\cC\slash \cA_5}^*} \longrightarrow \Sym^2 \Omega^1_{\cC}.
\]

At a point $JX$ where $X = \{F=0\}$, we have the following identifications:  ${\Omega^1_{\cA_5}}|_{\cC, JX} \cong \Sym^2R^1_F$, $\Omega^1_{\cC, JX} \cong R^2_F$, and the map $q$ is the multplication  $f\colon \Sym^2R^1_F \rightarrow R^2_F$ of classes of polynomials. Hence ${N_{\cC\slash \cA_5, JX}^*}  \cong J^2_F$, so the composition of the second fundamental form with the multiplication map of classes of polynomials is an endomorphism
\begin{equation} \label{mII}
\left(R^1_F\right)^*  \cong J^2_F \xrightarrow{II_{\cC \slash \cA_5}}
   \Sym^2R^2_F   \xrightarrow{m_F}  R^4_F \cong \left(R^1_F\right)^*.
  \end{equation}
  
Our main result is the following (see Theorem~\ref{main}).

\begin{thm} 
The composition $m_F\circ II_{\cC \slash \cA_5}$ is identically zero. Equivalently, the image of the second fundamental form $II_{\cC \slash \cA_5}$ is contained in the kernel of the multiplication map. \end{thm}

Notice that we have an unexpected symmetry, which deserves further investigation: namely, the second fundamental form at a point $JX$ is a map between two kernels of multiplication maps: 
$$
\Ker\left(f\colon \Sym^2R^1_F \lra R^2_F\right) \lra \Ker\left(m_F\colon \Sym^2R^2_F \lra R^4_F\right).
$$

We recall that the monodromy group of the family of cubic threefolds tensored with ${\mathbb Q}$ is the whole symplectic group (see \cite[Theorem 4]{beauville_mono}, see also \cite{huybrechts_notes}). This, in particular,  implies that  the second fundamental form $II_{\cC \slash \cA_5}$ is different from zero. 
Another way to see this is by observing that the special Mumford--Tate group for the general cubic threefold is the whole symplectic group. Hence the closure $\overline{\mathcal C}$ of~${\mathcal C}$ in ${\mathcal A}_5$ is not a special subvariety. Since ${\mathcal C}$ contains $CM$ points, by a result of Moonen \cite{moonen}, $\overline{\mathcal C}$ is not totally geodesic (see Remark~\ref{totgeo}). 

The proof proceeds by a series of steps,  some of them of possible independent interest. In the spirit of Mumford and Beauville, we take a line $l$ of the Fano variety, and we consider the conic bundle structure induced on $X$ by $l$. The discriminant curve is a plane quintic $Q_l$ naturally endowed with an odd non-trivial $2$-torsion line bundle $\alpha_l$, which means that $h^0(Q_l, {\mathcal O}_{Q_l}(1) \otimes \alpha_l) =1$. So a non-trivial section $s \in H^0(Q_l, {\mathcal O}_{Q_l}(1) \otimes \alpha_l)$ has a zero divisor $D$ such that $2D$ is the intersection of a conic $C$ with $Q_l$. Then the intermediate Jacobian of $X$ is isomorphic as ppav to the Prym variety $P(Q_l,\alpha_l)$. If $l$ is generic, a natural Del Pezzo surface $S$ appears in the picture, which is the complete intersection of the polar quadrics of the points of $l$. These quadrics give a $2$-dimensional vector subspace $I_2(\omega_{Q_l}\otimes \alpha_l)\subset J^2_F$, which can be identified with the set of quadrics containing the Prym canonical image of $Q_l$.

One of the key points in the proof is to show that the Hodge-Gaussian map introduced in \cite{cpt} gives a lifting of the restriction of the second fundamental form to $I_2(\omega_{Q_l}\otimes \alpha_l)$ (see Theorem~\ref{rho}). Since the composition of this map with a suitable multiplication map is the second Gaussian map $\mu_2$ of $\omega_{Q_l}\otimes \alpha_l$ on the plane quintic $Q_l$, we focus on the study of this map. We prove that the image is contained in the Jacobian ideal of the quintic (see Section~\ref{sec4} for precise statements). This result allows us to prove that the composition in \eqref{mII} is a multiple of the identity (see Theorem~\ref{identity}).

Next, we use the surface $S$ to connect the study of the Jacobian rings of $X$ and $Q_l$. This is done in Section~\ref{sec5}, where a further lifting of the second Gaussian map $\mu_2$ to $R^4_F$ is constructed, so that we get a map $I_2(\omega_{Q_l} \otimes \alpha_l) \rightarrow R^4_F$, which is the composition of the map $\mu_2\colon I_2(\omega_{Q_l} \otimes \alpha_l) \rightarrow H^0(Q_l, \omega_{Q_l}^{\otimes 4})$ with a map $\tilde{\tau}\colon H^0(Q_l, \omega_{Q_l}^{\otimes 4}) \rightarrow R^4_F$. 

We are able to make explicit computations, showing that the map $\tilde{\tau}$ vanishes, when the $2$-torsion point $\alpha_l$ on the quintic is of the form $2x-y-z$ for some points $x,y,z\in Q_l$ and $Q_l$ is smooth. In other words, the conic $C$ is the union of two lines $l_1$ and $l_2$ such that they intersect at a point in $Q_l$ and $l_1$ is tangent at order~$4$ at another point. In Proposition~\ref{rank3} we prove that this is equivalent to the geometric property that there is a point in $l$ such that its polar quadric has rank $3$. 
Then we show that there exists a non-empty open subset of the moduli space ${\mathcal C}$ of cubic threefolds $X$ such that $X$ contains a point whose polar quadric is of rank $3$ and a line $l$ containing this point, with the quintic $Q_l$ smooth (see Remark~\ref{adler+}). This can be seen as a complement of the results of Adler \cite{ar} on the singularities of the Hessian of $X$. 

Finally, we exhibit explicit constructions of such polar quadrics of rank $3$, where we are able to compute the second Gaussian map $\mu_2$ and show that the above lifting $\tilde{\tau} $ is identically zero. 

This, together with the fact that the composition $m_F \circ II_{\cC \slash \cA_5}$ in \eqref{mII} is a multiple of the identity, will allow us to conclude that $  m_F \circ II_{\cC \slash \cA_5}$ is zero.

The structure of the paper is as follows: 
In Section~\ref{sec2} we set the notation, and we recall the classical known results on the geometry of cubic threefolds. In Section~\ref{sec3} we analyze the relation between the period map associating with a cubic threefold $X$ its intermediate Jacobian and the Prym map of the corresponding plane quintics $Q_l$, when $l$ varies in the Fano surface of $X$. In Section~\ref{sec4} we recall the definition of Gaussian maps, we describe the second Gaussian map $\mu_2$ of an odd Prym canonical line bundle $\omega_Q \otimes \alpha$ on a smooth plane quintic $Q$, and we show that its image is contained in the Jacobian ideal of $Q$ (see Proposition~\ref{mu2} for a precise statement). In Section~\ref{sec5} we prove that the second Gaussian map $\mu_2$ can be lifted to a map $I_2(\omega_Q \otimes \alpha) \rightarrow R^4_F$, where $X = \{F=0\}$ (Proposition~\ref{lift}). The proof of Proposition~\ref{lift} is based on the description of the relation between the Jacobian ideals of $X$ and $Q$ and uses the Del Pezzo surface $S$. In Section~\ref{sec6} we study the second fundamental form of the moduli space of cubic threefolds embedded in ${\mathcal A}_5$ via the period map. We show that the Hodge-Gaussian map $\rho$ gives a lifting to its restriction to $I_2(\omega_Q \otimes \alpha)$; hence when we vary the line $l$, the Hodge-Gaussian maps determine the second fundamental form (Theorem~\ref{rho} and Remark~\ref{rho+}). Finally, we show that the composition of the second fundamental form with the multiplication map $m_F \colon \Sym^2 R^2_F \rightarrow R^4_F$ is a multiple of the identity (Theorem~\ref{identity}). In Section~\ref{sec7} we recall some results of Adler on polar quadrics to~$X$ of rank at most $3$, and we show that there exists a non-empty Zariski open subset of the moduli space of cubic threefolds $X$ such that $X$ contains a point whose polar quadric is of rank $3$ and a line $l$ containing this point such that the corresponding quintic $Q_l$ is smooth (Theorem~\ref{dominant} and Remark~\ref{adler+}). In Section~\ref{sec8} we prove the main theorem (Theorem~\ref{main}) showing that the composition of the second fundamental form with the multiplication map is identically zero.

\subsection*{Acknowledgments}
We thank Alessandro Ghigi for a very useful conversation about the monodromy of totally geodesic subvarieties. We also thank Andr\'es Rojas for pointing out the work of Adler on cubic threefolds \cite{ar}. Finally, we thank the anonymous referee for the  useful comments and suggestions to improve the paper.

\section{Notation and preliminaries}\label{sec2}

In this paper $X\subset \mathbb P^4$ will be a smooth cubic threefold given by the set of zeroes of the reduced polynomial $F\in \mathbb C[x_0,\ldots ,x_4]$. The Fano surface of lines in $X$ will be denoted by $F(X)$. The polar variety of a point $p=[a_0:\dots :a_4]$ with respect to $X$ is the quadric
\[
\Gamma_p(X)=V\left(a_0 \frac{\partial F}{\partial x_0}+\dots +
a_4 \frac{\partial F}{\partial x_4}\right).
\]
The same definition works for the polar variety of a point $p\in \mathbb P^n$ with respect to any smooth hypersurface, in particular for plane curves.

A line $l\in F(X)$ is said to be special if there exists a $2$-plane $\Pi$ containing $l$ such that the residual conic is a double line. In other words, $\Pi \cdot X=l+2r$ for some line $r\subset X$. Otherwise, the line is said to be non-special. The special lines  define a curve on $F(X)$. Notice that this is not the definition of ``lines of second type'' (see \cite{cg}); with our notation, $r$ is a line of second type.

\subsection*{Prym theory}
We define the moduli space $\cR_g$ of the isomorphism classes of pairs $(C,\alpha)$, where $C$ is a smooth irreducible curve of genus $g$ and $\alpha$ is a non-trivial $2$-torsion line bundle on $C$. Denote by $\pi\colon \cR_g \rightarrow \cM_g$ the forgetful map. This is a finite morphism of degree $2^{2g}-1$.

There is an unramified irreducible covering $\pi\colon\widetilde C \rightarrow C$ attached to the data $(C,\alpha)$. The Prym variety of $(C,\alpha)$ is defined as the component of the origin of the kernel of the norm map: $P(C,\alpha):=\Ker(Nm_{\pi})^0$. This is an abelian variety of dimension $g-1$, and the polarization in $J\widetilde C$ induces on  $P(C,\alpha)$ twice a principal polarization. Hence a map $P_g\colon\mathcal R_g \rightarrow \mathcal A_{g-1}$ is well defined (see \cite{mu} for more details). 

We are mainly concerned with the genus $6$ case. In particular, we consider plane curves of degree $5$. We denote by $\mathcal Q \subset\cM_6$ the moduli space of smooth quintic plane curves as a subspace of the moduli space of curves. Then the preimage $\pi^ {-1}(\mathcal{Q})$ breaks into two irreducible components:
$$
\begin{aligned}
&
\mathcal R\mathcal Q^+:=\{(Q,\alpha)\in \cR_6 \mid Q\in \mathcal Q\text{ and } h^0(Q,\cO_{Q}(1)\otimes \alpha)\text{ even} \},
\\
&
\mathcal R\mathcal Q^-:=\{(Q,\alpha)\in \cR_6 \mid Q\in \mathcal Q\text{ and } h^0(Q,\cO_{Q}(1)\otimes \alpha)\text{ odd} \}.
\end{aligned}
$$
It is well known that the Prym map applies birationally $\cR\mathcal Q^+$ to the Jacobian locus in $\cA_5$.

\subsection*{Conic bundle structures on $\boldsymbol{X}$}
All the results can be found in \cite{cg}, \cite{be_jacob} and, especially, \cite{donsmith}.
Given a line $l\in F(X)$, we denote by $X_l$ the blow-up of $X$ at $l$. The projection $X\setminus\{l\} \rightarrow \mathbb P^2$ into a $2$-plane disjoint from $l$ extends to $X_l$. Since all the fibers are conics, this is called a conic bundle structure on $X$. The discriminant curve $Q_l\subset \mathbb P^2$ is a quintic parametrizing the singular conics. If the line is non-special, the quintic $Q_l$ is smooth. We assume this from now on. Moreover, the curve in the Fano surface $
\widetilde Q_l :=(\text{closure of})\, \{r\in F(X)\setminus \{l\} \mid r \cap l \neq \emptyset \}
$
is an irreducible curve of genus $11$ with a natural involution such that the quotient by this involution is the quintic~$Q_l$. This unramified covering $\widetilde{Q_l}\rightarrow Q_l$ is determined by a non-trivial $2$-torsion point $\alpha_l\in JQ_l[2]\setminus \{0\}$.  

We collect some well-known results and we fix some notation. 

\begin{prop}  \label{basic_properties}\leavevmode
\begin{enumerate}
    \item \label{bp-1} The $2$-torsion point $\alpha_l $ is odd; that is, $h^0(Q_l,\cO_{Q_l}(1)\otimes \alpha_l )=1$. In particular, $(Q_l,\alpha_l)\in \cR\mathcal{Q}^-$.
    \item \label{bp-2} The Prym variety $P(Q_l,\alpha_l)$ is isomorphic $($as ppav$)$ to the intermediate Jacobian $JX$.
    \item \label{bp-3} The fiber of the Prym map at $JX$ is exactly the open set of $F(X)$ given by the non-special lines; hence it is the complement of a curve.
    \item \label{bp-4} The map $\psi\colon Q_l \rightarrow X\subset \mathbb P^4$ sending a point $p\in Q_l$  to the double point of the corresponding degenerate conic is the map attached to the line bundle $\omega_{Q_l}\otimes \alpha_l$.
    \item \label{bp-5} Let $s\in H^0(Q_l,\cO_{Q_l}(1)\otimes \alpha _l )\cong \mathbb C$ be a non-trivial section, and put $(s)_0=:D=p_1 +\dots+p_5$,
      where the $p_i$ are not necessarily distinct. Nevertheless, for a generic line $l$, $D$ is reduced. Observe that $2D \in |\omega_{Q_l}|$, so there exists a unique conic $C$ such that $C \cdot Q_l = 2D$.  Since $\alpha_l$ is not trivial, the conic $C$ is reduced. Then  $l\cdot \psi (Q_l)=\psi(p_1) +\cdots + \psi(p_5)$. 
    Notice that the points $\psi(p_i)$ are the points $x$ in $l$ such that there is a plane $\Pi$ containing $l$ and such that $\Pi \cdot X$ is formed by three different lines passing through $x$. We call these points ``triple'' points of $X$.
    \item \label{bp-6} Denote by $N_D \subset H^0({\mathbb P}^2, \mathcal O_{{\mathbb P}^2}(3))$ the subspace of plane cubics osculating the conic $C$ to a given order in the points in $\Supp(D)$. In the case where $D$ is reduced, it is simply  the space of plane cubics containing the five points $p_i$.  Let $\varphi\colon \mathbb P ^2 \dasharrow \mathbb P^4$ be the map induced by $|N_D|$; then $\psi $ is the restriction to $Q_l$ of $\varphi$. 
    
     Indeed, assume that $D$ is reduced; then the diagram
    $$
\xymatrix@R=0.8cm@C=0.8cm{
             & &  0 \ar[d] & 0 \ar[d] & \\
    0 \ar[r]& \cO_{\mathbb P^2}(-2) \ar[d]^{=} \ar[r] & \mathcal I_{D\slash \mathbb P^2}(3)\ar[d] \ar[r] & \mathcal O_{Q_l}(3)(-D)\cong \omega_{Q_l}\otimes  \alpha_l \ar[d]\ar[r]& 0\\
    0 \ar[r]& \cO_{\mathbb P^2}(-2) \ar[r]^{\cdot Q_l} & \mathcal O_{\mathbb P^2}(3)\ar[d] \ar[r] & \mathcal O_{Q_l}(3) \ar[d]\ar[r]& 0\\
     && \mathcal O_D(3) \ar[r]^{=} \ar[d] & \mathcal O_D(3) \ar[d] &
    \\
    &&0&0&
    }           
$$
    induces  an isomorphism $H^0(\mathbb P^2,\mathcal I_{D\slash \mathbb P^2}(3))\cong H^0(Q_l,\omega_{Q_l}\otimes \alpha_l) $.
  
    \item \label{bp-7} The vector space $I_2(\omega_{Q_l}\otimes \alpha_l)$ of the equations of the quadrics containing $\psi (Q_l)$ has dimension $2$. Such quadrics are the polar varieties with respect to $X$ of the points in the line $l$. 
    \item\label{bp-8} The closure of the image of the map $\varphi$ is a surface $\bar{S}$, which is the complete intersection of two independent quadrics in $I_2(\omega_{Q_l}\otimes \alpha_l)$. In the case where $D$ is reduced,  $\bar{S}$ is the Del Pezzo surface given by the blow-up of ${\mathbb P}^2$ at the points $p_i$. If $D$ is not reduced,  
  then  the blow-up $S$ of the points and of the infinitely near points gives the minimal resolution of  the map $\varphi$: 
\begin{equation*}
\xymatrix@R=0.8cm@C=0.8cm{
   S \ar[d]^{\varepsilon}  \ar[dr]^{\overline {\varphi}}& \\
   \mathbb P^2 \ar@{-->}[r]^{\varphi } & \overline{S}\subset \mathbb P^4. 
   }           
\end{equation*}
\end{enumerate}
\end{prop}

\subsection*{Jacobian rings}
Let $F(x_0,\ldots ,x_n)=0$ be the equation of an irreducible reduced hypersurface $X$ of degree $d$ in $\mathbb P^n$. We use the following notation:
\[
\begin{aligned}
& S=\mathbb C[x_0,\ldots ,x_n],\qquad S^d=\mathbb C[x_0,\ldots ,x_n]_d, \\
& J_F=\left<\frac{\partial F}{\partial x_0}, \ldots ,\frac{\partial F}{\partial x_n}\right>\subset S \quad \text{(the Jacobian ideal),} \\
& S/J_F=R_F=\oplus_{i\ge 0} R_F^i \quad \text{(the Jacobian ring)}.
\end{aligned}
\]
Observe that $R_F^i=S^i$ for $i\le d-2$. Moreover, $R_F^{d}$ is the subspace of infinitesimal deformations (therefore, in $H^1(X,T_X)$) which deform $F$ as a hypersurface (attaching to a class of a degree $d$ polynomial $G$ the infinitesimal deformation $F+\varepsilon G, \varepsilon ^2=0$). 

According to Macaulay's theorem (see for example \cite{vo_book}), assuming that $X=V(F)$ is smooth, we have 
\[
R^i_F \neq 0 \iff 0\le i \le N=(d-2)(n+1);
\]
moreover, $\dim_{\mathbb C}R_F^N=1$, and the product of classes of polynomials provides a duality $ R_F^i\cong (R_F^{N-i})^*$. 
We will use these rings in two  cases: the cubic threefold $X=V(F)$ in $\mathbb P^4$ and the quintic plane curve $V(Q_l)$. In the first case, $N=5$, and $R_F^3$ is the space of  infinitesimal deformations as cubic threefold. In the second case, $N=9$, and $R_{Q_l}^5$ is the space of infinitesimal deformations as a plane curve.

\section{Cubic threefolds and quintic plane curves}\label{sec3}

Let $\mathcal C\subset \mathcal A_5$ be the locus of the intermediate Jacobians of smooth cubic threefolds in $\mathbb P^4$. By the Torelli theorem, this space can be identified with the moduli space of cubic threefolds. As recalled in Proposition~\ref{basic_properties}\eqref{bp-2}, the restriction of the Prym map gives a surjection $\mathcal {RQ}^- \rightarrow \mathcal C$. For any $JX\in \mathcal C$, the fiber of the Prym map is an open set of the Fano surface $F(X)$ of $X$.

We give some diagrams involving the differential of $P_6$ and of its restriction $P$ to $\mathcal {RQ}^-$.   We set $T_F := \Ker (dP\colon T_{\mathcal {RQ}^-} \rightarrow P^*T_{\mathcal C})$. Then we have the following diagrams of (orbifold) vector bundles over $\mathcal {RQ}^-$:
\begin{equation}\label{diagrama1}
\xymatrix@R=0.7cm@C=0.7cm{
              &  0 \ar[d] & 0 \ar[d] && \\
    &  T_F \ar@{=}[r] \ar[d] & T_F \ar[d] & 0\ar[d]& \\
    0 \ar[r] & T_{\mathcal {RQ}^-} \ar[r] \ar[d] & {T_{\cR_6}}|_{\mathcal {RQ}^-} \ar[d] \ar[r] & N_{\mathcal {RQ}^-\slash \cR_6} \ar[r] \ar[d]& 0
    \\
    0 \ar[r] & P^*T_{\cC} \ar[r] \ar[d]  & P^*{T_{\cA_5}}|_{\cC} \ar[d] \ar[r] & P^*{N_{\cC\slash \cA_5}} \ar[r] \ar[d]& 0 
    \\
    &0 & \mathcal V \ar@{=}[r] \ar[d] & \mathcal V \ar[d] &
    \\
    &&0&0\rlap{.}&&
    }           
\end{equation}

We consider a generic smooth cubic threefold and a non-special line $l\in F(X)$. Then, $(\widetilde Q_l,Q_l) \in \mathcal {RQ}^-$. We take the stalks at $l$ in diagram \eqref{diagrama1}, and we dualize.
Using the identifications
\[
T_{\cA_5\vert \cC , JX}^*=\Sym^2H^0\left(Q_l,\omega_{Q_l}\otimes \alpha_l\right), \qquad
T_{\cR_6,l}^*=H^0\left(Q_l,\omega_{Q_l}^{\otimes 2}\right),    
\]
we get a diagram of vector spaces as follows:
 \begin{equation}\label{diagrama2}
\xymatrix@R=0.7cm@C=0.7cm{
              &  0 \ar[d] & 0 \ar[d] && \\
    &  \mathcal V_l^* \ar@{=}[r] \ar[d] & \mathcal{V}_l^* \ar[d] & 0\ar[d]& \\
      0 \ar[r] &        N_{\cC\slash \cA_5, JX}^* \ar[r] \ar[d] & \Sym^2H^0\left(Q_l,\omega_{Q_l}\otimes \alpha_l\right)  \ar[d]^{m_{\alpha }} \ar[r]^{\hspace{10mm}f} & T_{\cC, JX}^* \ar[r] \ar[d]& 0
    \\
   0 \ar[r] & N_{\mathcal {RQ}^-\slash \cR_6,l}^* \ar[r] \ar[d]  & H^0\left(
   Q_l,\omega_{Q_l}^{\otimes 2}\right)        \ar[d] \ar[r] &  T_{\mathcal {RQ}^-,l}^*   \ar[r] \ar[d]& 0 
        \\
    &0 & T_{F(X),l}^* \ar@{=}[r] \ar[d] & T_{F(X),l}^* \ar[d] &
    \\
    &&0&0\rlap{.}&&
    }      
\end{equation}

Notice that $\mathcal V_l^*$ can be seen as the $2$-dimensional vector space of the quadrics in $\mathbb P(H^0(Q_l,\omega_{Q_l}\otimes \alpha_l))^*$ that contain the image of the semicanonical map $\psi$ (see Proposition~\ref{basic_properties}\eqref{bp-6} and~\eqref{bp-7}). Therefore, 
$\mathcal V_l^*=I_2(\omega_{Q_l}\otimes \alpha_l)$. 

Remember that
\[
 R_F^1=S^1=H^0\left(Q_l,\omega_{Q_l}\otimes \alpha_l\right), \quad 
T_{\cC ,JX}^*=\left(R_F^3\right)^*=R_F^2.
\]
Hence the map $f$ in the second row  is the multiplication map $\Sym^2R^1_F \rightarrow R^2_F$, and then  
\[
N_{\cC\slash \cA_5, JX}^*=J^2_F.
\]
Also $H^0(Q_l,\omega_{Q_l}^{\otimes 2})=H^0(\mathbb P^2, \cO_{\mathbb P^2}(4))=S^4$ and $T_{\mathcal {RQ}^-,l}^*=T_{\mathcal Q, Q_l}^*=(R^5_{Q_l})^*=R^4_{Q_l}$. Therefore, 
\[
N_{\mathcal {RQ}^- \slash \cR_6,l}^*=J^4_{Q_l}.
\]

With the above identifications, diagram \eqref{diagrama2} becomes 
\begin{equation}\label{diagramma3}
\xymatrix@R=0.7cm@C=0.7cm{
              &  0 \ar[d] & 0 \ar[d] && \\
    &  I_2\left(\omega_{Q_l}\otimes \alpha_l\right) \ar@{=}[r] \ar[d] & I_2\left(\omega_{Q_l}\otimes \alpha_l\right) \ar[d] & 0\ar[d]& \\
      0 \ar[r] &        J^2_F \ar[r] \ar[d] & \Sym^2H^0\left(Q_l,\omega_{Q_l}\otimes \alpha_l\right)\cong H^0\left({\mathbb{P}}^4, {\mathcal O}_{{\mathbb P}^4}(2)\right)  \ar[d]^{m_{\alpha}} \ar[r]^{\hspace{30mm} f} & R^2_F \ar[r] \ar[d]& 0
    \\
   0 \ar[r] & J^4_{Q_l} \ar[r] \ar[d]  & H^0\left(Q_l,\omega_{Q_l}^{\otimes 2}\right) \cong  H^0\left({\mathbb{P}}^2, {\mathcal O}_{{\mathbb P}^2}(4)\right)      \ar[d] \ar[r]^{\hspace{30mm} f'} &  R^4_{Q_l}   \ar[r] \ar[d]& 0 
        \\
    &0 & T_{F(X),l}^* \ar@{=}[r] \ar[d] & T_{F(X),l}^* \ar[d] &
    \\
    &&0&0\rlap{.}&&
    }      
\end{equation}

We consider the natural isomorphisms
\[
\left(R^1_F\right)^ *\lra J^ 2_F, \qquad \left(R^ 1_{Q_l}\right) ^*\lra J^ 4_{Q_l}
\]
 given by the polarity; that is, a vector $w\in (R_F^1)^*$ with coordinates $(a,b,c,d,e)$ (in other words, representing the point $[a:b:c:d:e]\in \mathbb P (R^1_F)^*=\mathbb P^4$) is sent to
 \[
 aF_{x_0}+bF_{x_1}+cF_{x_2}+dF_u+eF_v,
 \]
 and similarly for $Q_l$.

Then the dual of the exact sequence in the first column of diagram \eqref{diagramma3} has a natural interpretation, as follows. 

\begin{lem}\label{lemma_s}
Let $s\in H^0(Q_l,\cO_{Q_l}(1)\otimes \alpha_l)$ be a non-trivial section; then multiplication by $s$ gives a short exact sequence
\[
0\lra R^1_{Q_l}=H^0\left(Q_l,\cO_{Q_l}(1)\right) \xrightarrow{\;\cdot s\;}
R^1_F=H^0\left(Q_l,\omega_{Q_l}\otimes \alpha_l\right)
\lra
I_2(\omega_{Q_l}\otimes \alpha_l)^* \lra 0.
\]
\end{lem}

\begin{proof}
We think of $R^1_F$ as the vector space of the  equations of hyperplanes in $\mathbb P^4$. Then the image of the multiplication by $s$ is the vector subspace $V_l$ of linear forms vanishing on $l$. So we have  
\[
0\lra V_l  \lra R^1_F \lra
H^0(l,\cO_l(1))\lra 0,
\]
the second map being the restriction to $l$. The dual gives the inclusion $H^0(l,\cO_l(1))^*\subset (R^ 1_F)^*$. By  Proposition~\ref{basic_properties}\eqref{bp-7}, the polar of the points in $l$ give the quadrics in $I_2(\omega_{Q_l}\otimes \alpha_l)$; hence we are done.
\end{proof}

\begin{rem}
Notice that the first vertical exact sequence 
$$
0 \lra I_2\left(\omega_{Q_l} \otimes \alpha_l\right) \lra J^2_F \lra J^4_{Q_l} \lra 0
$$
appears in three different ways: 
\begin{enumerate}
\item as the conormal exact sequence in the diagram,
\item as the dual of the multiplication by $s$ as in Lemma~\ref{lemma_s},
\item as the map sending the polar quadric to the threefold of a point $p \in {\mathbb P}^4$ to the polar quartic to the quintic of the projection of the point $p$ in ${\mathbb P}^2$. This can be easily seen from diagram \eqref{diagramma3}. 
\end{enumerate}

\end{rem}

\section{Second Gaussian map of plane quintics with an odd 2-torsion line bundle}\label{sec4}

In this section we fix a general cubic threefold $X$ with equation $F$ and a line $l\in F(X)$. Let $Q:=Q_l$ be the associated quintic plane curve, and let $\alpha :=\alpha_l $ be the non-trivial $2$-torsion point in the Jacobian $JQ$ such that $P(Q,\alpha)\cong JX$. We will assume that $l$ is non-special and that $Q$ is smooth.

To define the second Gaussian maps attached to a line bundle $L$ on $Q$, we consider on the surface $Q\times Q$ the line bundle $M:=p_1^*(L)\otimes p_2^*(L)$. Let $\Delta \subset Q\times Q$ be the diagonal, and consider the restriction maps
\[
\widetilde {\mu_{n, L}}\colon H^0(Q\times Q, M(-n\Delta))\lra H^0\left(Q\times Q, M(-n\Delta)\vert_{\Delta}\right).
\]
Since $\mathcal O_{Q\times Q} (\Delta)\vert_{\Delta}\cong \omega_Q^{-1}$, then 
\[
H^0\left(Q\times Q, M(-n\Delta\right)\vert_{\Delta})\cong H^0\left(Q,L^{\otimes 2}\otimes \omega_Q^{\otimes n}\right).  
\]
So, for $n=2$, we can define the second Gaussian map 
\begin{equation}
    \label{mu2L}
\mu_{2, L}\colon I_2(L) \lra H^0\left(Q, L ^{\otimes 2} \otimes \omega_Q^{\otimes 2}\right)
\end{equation}
as the restriction of $\widetilde {\mu_{2, L}}$ to $I_2(L) $, where $I_2(L): = \Ker (\Sym^2 H^0(Q, L) \rightarrow H^0(Q,L ^{\otimes 2})) $. We are interested in the map 
\[
\mu_2:=\mu_{2, \omega_Q \otimes \alpha}\colon I_2(\omega_Q \otimes \alpha) \lra H^0\left(Q,   \omega_Q^{\otimes 4}\right).
\]

For $n =1$, we have $\Lambda^2(H^0(Q, L)) \subset H^0(Q\times Q, M(-\Delta))$, and we consider the restriction of $\widetilde {\mu_{1, L}}$ to  $\Lambda^2(H^0(Q, L))$, 
$$
\mu_{1,L} \colon \Lambda^2(H^0(Q, L)) \lra H^0\left(Q, L^{\otimes 2} \otimes \omega_{Q}\right),
$$
given by 
$$s_1 \wedge s_2 \longmapsto s_2^2  \ d \left(\frac{s_1}{s_2}\right).$$

\begin{rem}
\label{tau}
There is a natural map $\tau \colon H^0(Q,\mathcal O_Q(8))\cong H^0(Q,\omega ^{\otimes 4}_Q) \rightarrow R^8_Q$. 

Indeed, there is a short exact sequence
\[
\xymatrix@R=0.7cm@C=0.7cm{
               0 \ar[r] & S^3\ar[r]^{\hspace{-15mm}\cdot Q} & S^8=H^0(\mathbb P^2, \mathcal O_{\mathbb P^2}(8)) 
              \ar[d] \ar[r] & H^0\left(Q,\omega_Q^{\otimes 4}\right) \ar[r] \ar[dl]^{\tau} & 0.
              \\
      & & R^8_Q  && } 
\]
Since $Q\cdot S^3\subset S^8$ belongs to the Jacobian ideal of $Q$,  there is a well-defined map $\tau\colon H^0(Q,\omega_Q^{\otimes 4})\rightarrow R^8_Q$.
\end{rem}

Our aim is to prove the vanishing of the composition of $\mu_2$ with $\tau$.

\begin{prop}
\label{mu2}
The map $\mu_2$ is injective, and 
\[
\tau \circ \mu_2\colon   I_2(\omega_{Q}\otimes \alpha) \longrightarrow R^8_Q
\]
is identically zero.
\end{prop}
\begin{proof} 
It suffices to prove the result in the generic case where $D$ is reduced. 
We compute $\mu_2$ on the rank~$4$ quadrics $\Gamma_i$ corresponding to the polar curves of the five points $\psi (p_i)$. To this purpose, we fix one of  them, $p_i \in Q$, and we consider the pencils $|L_i|:=|\cO_Q(1)(-p_i)|$ and $|M_i|=|\cO_Q(1)(p_i)\otimes \alpha| = |\omega_{Q} \otimes \alpha \otimes L_i^{-1}|$. Taking bases $\{s_1, s_2\}$ of $H^0(Q, L_i)$ and  $\{t_1, t_2\}$ of $H^0(Q, M_i)$, then $\Gamma_i = s_1t_1 \odot s_2t_2 - s_1t_2 \odot s_2t_1$ and $\mu_2(\Gamma_i)=\mu_{1, L_i} (s_1 \wedge s_2)\mu_{1, M_i} (t_1 \wedge t_2)$ (see,  \textit{e.g.}, \cite[Lemma~2.2]{cf-michigan}), where we denote by $\odot$ the symmetric product, $v \odot w := v \otimes w + w \otimes v$. The linear system attached to $L_i$ is the $g^1_4$ given by the lines through $p_i$, and the divisor $\div(\mu_{1,L_i}(s_1\wedge s_2))=:R_i$ is the degree $18$ divisor of points $x\in Q$ such that the tangent line at $x$ goes through $p_i$, so $R_i = \div(\Gamma_{p_i}(Q) ) -2p_i$. 

Generically, this divisor does not contain any of the five points. On the other hand, the line bundle
$M_i\cong\cO(2)(-p_1-\dots -\widehat {p_i}-\dots -p_5)$ can be seen as the linear system of conics through the rest of the points $p_j$. The divisor $\div(\mu_{1,M_i}(t_1\wedge t_2))=:D_i$ is given by the points $x\in Q$ such that there is a conic passing through $p_j$, $j\neq i$, and tangent to $Q$ at $x$. Since the conic $C$ determined by $p_1,\dots ,p_5$ satisfies $Q\cdot C=2p_1+\dots +2p_5$, we have that $p_i\in D_i$ and, for a general $Q$, the rest of the points $p_j$ do not appear in $\mu_2(\Gamma_i)$. 
This shows that $\mu_2$ is injective since the divisors of $\mu_2(\Gamma_i)$ and $\mu_2(\Gamma_j)$, $i\ne j$, are different and $ I_2(\omega_{Q}\otimes \alpha) $ has dimension $2$. 

Now we put $D_i=D_i'+p_i$, and we denote by $\mu_2 (\Gamma _i) =\rho_i
\in H^0(Q, \mathcal O_{Q}(8))$, $i=1,\ldots, 5$, the images of the
five elements $\Gamma _i\in I_2(\omega_{Q}\otimes \alpha)$
corresponding to the rank $4$ quadrics. We choose polynomials $H_i\in
S^8$ representing $\rho_i\in S^8/Q\cdot S^3$. Then, according to the
previous discussion, 
\[
\div\left(\rho_i\right)=\div\left(\Gamma_{p_i}(Q)\right)-2p_i +D_i'+p_i=\div\left(\Gamma_{p_i}(Q)\right)-p_i +D'_i,
\]
which is an effective divisor of degree $40$. 
Let $l_{1i},l_{2i}\in H^0(Q,\mathcal O_{Q}(1)(-p_i))$ be the equations of two different lines passing through $p_i$. Then
\[
\div\left(l_{ji}\cdot \rho_i\right)=\div \left(\Gamma_{p_i}(Q)\right)+D_i'+E_{ji},
\]
where $j=1,2$ and $E_{ji}$ is an effective divisor such that $\div(l_{ji})=p_i+E_{ji}$. Therefore, $D'_i+E_{ji}$ is the effective divisor attached to a section in $ H^0(Q, \mathcal O_{Q}(5))$. Notice that the short exact sequence 
$$
0 \longrightarrow \mathcal O_{\mathbb P^2} \longrightarrow                  \mathcal O_{\mathbb P^2}(5) \longrightarrow    
  \mathcal O_{Q}(5) \longrightarrow 0
$$
implies that $H^0(\mathbb P^2, \mathcal O_{\mathbb P^2}(5))$ surjects onto $H^0(Q, \mathcal O_{Q}(5))$. Hence there are  homogeneous polynomials $G_1, G_2$ of degree $5$ that, up to multiples of the equation of $Q$, satisfy  $l_{ji}\cdot H_i=G_j \cdot \Gamma_{p_i}(Q)\in J^9_{Q}$.
 Now we consider the line $r_{lk}$ determined by any two points $p_l, p_k$ with $i\neq l,k$. If we prove that  $r_{lk}\cdot H_i$ also belongs to $J^9_{Q}$, then  since $l_{1i}, l_{2i}, r_{lk}$ generate $R^1_{Q}$, we obtain that $R_Q^1 \cdot H_i=0$ in $R^9_{Q}$, so $H_i\in J^8_{Q}$ since the multiplication in the Jacobian ring is a perfect pairing.  

Since $I_2(\omega_{Q}\otimes \alpha)$ is $2$-dimensional, we can write $\rho_i$ as a linear combination $\rho_i=\lambda_l \rho_l + \lambda_k \rho_k $.  Then we consider $r_{lk} \cdot (\lambda_l H_l + \lambda_k H_k)$. By the 
previous discussion, since $r_{lk}$ contains $p_l$ and $p_k$, we know that $H_l\cdot r_{lk}$ and $H_k \cdot r_{lk}$ belong to $J^9_Q$. Thus, so does $r_{lk}\cdot H_i \in J^9_Q$, and we are done.
\end{proof}

\section{A lifting of \texorpdfstring{$\boldsymbol{\tau}$}{Lg}}\label{sec5}

Throughout this section we will always assume that the (smooth) quintic $Q$ and the conic $C$ only intersect at non-singular points of $C$ and that the maximal intersection multiplicity of $Q$ and $C$ at a point is $4$. 

Consider the dual of the exact sequence in Lemma~\ref{lemma_s}. Via the identifications ${R^1_F}^* \cong R^4_F$ and ${R^1_Q}^* \cong R^8_Q$, it becomes 
$$
0 \longrightarrow I_2(\omega_Q \otimes \alpha) \longrightarrow  R^4_F \stackrel{h}\longrightarrow R^8_Q \longrightarrow 0.
$$
Our purpose is to show the following. 

\begin{prop}
\label{lift}
The map $\tau \colon H^0(Q,\omega ^{\otimes 4}_Q) \rightarrow R^8_Q$ lifts to a map $\tilde{\tau}\colon H^0(Q,\omega ^{\otimes 4}_Q) \rightarrow R^4_F$ such that $h\circ \tilde{\tau} = \tau$.
\end{prop}

To prove  Proposition~\ref{lift}, we need to introduce some preliminary results. 

Recall that the map $\psi\colon Q \rightarrow \mathbb P H^0(Q,\omega_Q \otimes \alpha)^*$ extends to the rational map $\varphi\colon {\mathbb P}^2 \dashrightarrow {\mathbb P}^4$ defined by the linear system $|N_D|$ (see   Proposition~\ref{basic_properties}\eqref{bp-6}). Consider the  surface $S$ defined  in  Proposition~\ref{basic_properties}\eqref{bp-8}.

Denote by $E$ the exceptional divisor of the map  $\epsilon$. Set $\mathcal O_S(M) ={\overline{\varphi}}^*(\mathcal O_{{\mathbb P}^4}(1))$, $\mathcal O_S(H) = \epsilon^*(  O_{\mathbb P^2}(1)    )$. Hence we have $\mathcal O_S(M) =\mathcal O_{S}(3H-E)$.

%In the case $D$ reduced we denote by $E_i$ the image in $S$ of the five exceptional divisors. Then the ample divisor $M$ defining the map can be written as $3H-E_1-\dots - E_5$, where $H$ is the pullback of $\mathcal O_{\mathbb P^2}(1)$ to the blow-up. Hence also $\mathcal O_S(M)=\mathcal O_S(3H-E)$, where $E:=E_1+\dots +E_5$.  

Consider the first two columns of diagram \eqref{diagramma3}. The map $m_{\alpha}\colon H^0(\mathbb P^4, {\mathcal O}_{\mathbb P^4}(2)) \rightarrow H^0(\mathbb P^2,{\mathcal O}_{{\mathbb P}^2}(4))$ factors through the restriction $H^0(\mathbb P^4, {\mathcal O}_{{\mathbb P}^4}(2)) \rightarrow H^0(S, {\mathcal O}_S(2M))$. We have $H^0(S,\mathcal O_S(2M-\psi(Q)))=H^0(S,\mathcal O_S(H-E))=0$ since there is no line passing through $D$,  so the restriction map $H^0(S,\mathcal O_S(2M))\rightarrow H^0(Q,\omega_Q^{\otimes 2}) \cong H^0({\mathbb P}^2, {\mathcal O}_{{\mathbb P}^2}(4))$ is injective.

\begin{rem}
\label{quartic-sextic}
Notice that the vector space $H^0(S,\mathcal O_S(2M)) = H^0(S,\mathcal O_S(6H-2E))$ is naturally identified with the space of the plane sextics osculating the conic at the right order. If $D$ is reduced, this means that they are singular in the five points $p_i$.  The inclusion $H^0(S,\mathcal O_S(2M))\rightarrow H^0(Q,\omega_Q^{\otimes 2}) \cong H^0({\mathbb P}^2, {\mathcal O}_{{\mathbb P}^2}(4))$ can be seen as follows: the divisor of the restriction of such a sextic to $Q$ is the divisor of the restriction of a unique plane quartic to $Q$ plus $2D$.  
\end{rem}

The surface $S$  will be useful to connect the study of the Jacobian rings of $X$ and $Q$.  

We can choose coordinates  in ${\mathbb P}^4$ such that $J^2_F$ is generated by five quadrics $F_{x_0}$, $F_{x_1}$, $F_{x_2}$, $F_u$, $F_v$, where $u,v$ are  coordinates on the line $l$ and $x_0$, $x_1$, $x_2$ are coordinates on the plane. The surface $\overline{S}$ is then the intersection of $F_u$ and $F_v$, and we set $L_i := \overline{\varphi }^*({F_{x_i}}) \in H^0(S, {\mathcal O}_S(2M))$ (see Proposition~\ref{basic_properties}\eqref{bp-8}).

Hence each $L_i$ corresponds to a plane sextic $T_i$ such that the restriction to $Q$ yields the unique quartic~$Q_{x_i}$.

\begin{lem}
\label{Li-Ti}
We have $T_i= CQ_{x_i}-QC_{x_i}$; equivalently, their strict transforms are the curves $L_i \in H^0(S,\mathcal O_S(2M))$. 
\end{lem}
\begin{proof}
 By the definition of the curves $T_i$, all the points $p_k$ in the support of $D$  belong to the three curves. On the other hand,
 \[
 \left(CQ_{x_i}-QC_{x_i}\right)_{x_j}=C_{ x_j} Q_{x_i}+C Q_{x_i\, x_j}-Q_{x_j}  C_{ x_i}-Q C_{x_i\, x_j}.
 \]
Since $Q$ and $C$ are tangent at the points $p_k$, we have that, for all $k$,  
\[
 (Q_{x_0}(p_k),Q_{x_1}(p_k),Q_{x_2}(p_k)) \wedge (C_{x_0}(p_k),C_{x_1}(p_k),C_{x_2}(p_k))=0.
\]
Hence $C_{ x_j} Q_{x_i}-Q_{x_j}  C_{ x_i}$ vanishes at all the points $p_k$, so $T_i$ is singular at the points $p_k$. If two of the $p_k$ are not distinct, we make a local computation to show that the osculating conditions of the sextic with the conic $C$ as in Remark~\ref{quartic-sextic} are  verified. 
Indeed, assume that the intersection multiplicity of $Q$ and $C$ at a point $p$ is $4$, and take coordinates so that $p = [1,0,0]$. We assume that the common tangent to the conic and the quintic in $p$ is the line $\{x_1 =0\}$. So we have 
$$Q = x_0^4 x_1 + x_0^3 g_2(x_1, x_2) + x_0^2 g_3(x_1, x_2) + x_0 g_4(x_1, x_2) + g_5(x_1, x_2),$$
where the $g_i(x_1, x_2)$ are homogeneous polynomials of degree $i$. First assume  that $C$ is smooth; then we may suppose that $C = x_0x_1 - x_2^2$. A straightforward computation shows that $Q(1, x_2^2, x_2) = x_2^2(1 +a_3) + x_2^3( a_2 + b_4) + h.o.t.$, where $a_2$, $a_3$, $b_4$ are the following coefficients of $g_2$ and $g_3$: 
$g_2(x_1, x_2) = a_1x_1^2 + a_2 x_1 x_2 + a_3 x_2^2$, $g_3(x_1, x_2) = b_1x_1^3 + b_2 x_1^2 x_2 + b_3 x_1 x_2^2 + b_4 x_2^3$.

So the conditions that the intersection multiplicity of $C$ and $Q$ in $p$ is $4$ are: $a_3 = -1$, $a_2 + b_4 =0$. 
Then one immediately checks that $(CQ_{x_i}-QC_{x_i})(1,0,x_2) = \lambda_i x_2^4 + h.o.t.$ for   $i =0,1,2$. Hence the conditions of Remark~\ref{quartic-sextic} are satisfied.

If $C$ is not smooth, we can assume  $C = x_0 x_1$, so in this case the conditions that the intersection multiplicity of $C$ and $Q$ in $p_1$ is $4$ are: $a_3 = b_4 =0$. Then again it is immediate to verify that  $(CQ_{x_1}-QC_{x_1})(1,0,x_2) = \lambda x_2^4 + h.o.t.$, while $(CQ_{x_i}-QC_{x_i})(1,0,x_2)=0$ for $i =0,2$. So in this case too, the conditions of Remark~\ref{quartic-sextic} are satisfied.

Since the divisor of $CQ_{x_i}-QC_{x_i}$ restricted to $Q$ is $2D + \div(Q_{x_i |Q})$, the image of $L_i$ in $H^0(Q,\omega_Q^{\otimes 2})$ coincides with $({CQ_{x_i}-QC_{x_i}})|_{Q}$. 
\end{proof}

Observe that, as a consequence of the Euler formula, the following equality holds:
\begin{equation} \sum x_iT_i= \sum x_i(CQ_{x_i}-QC_{x_i}) = C \sum x_i Q_{x_i}- Q \sum x_i C_{x_i}= 3CQ. \label{euler} \end{equation}

Now we consider the following exact sequence attached to the divisors $L_i$:
\begin{equation}\label{koszul_F}
  0 \lra {\mathcal O}_S(-2M) \lra {\mathcal O}_S^{\oplus 3} \lra {\mathcal O}_S^{\oplus 3}(2M) \xrightarrow{(L_0,L_1,L_2)}  {\mathcal O}_S(4M) \lra 0.
    \end{equation}
%
%
%\[
%0 \rightarrow {\mathcal O}_S(-2M) \rightarrow {\mathcal %O}_S^{\oplus 3} \rightarrow {\mathcal O}_S^{\oplus 3}(2M) %\stackrel{(L_1, L_2,L_3)}{\longrightarrow} {\mathcal O}_S(4M) %\rightarrow 0
%
%\]

This splits into the following two short exact sequences:
$$0 \longrightarrow {\mathcal O}_S(-2M) \longrightarrow {\mathcal O}_S^{\oplus 3} \longrightarrow \mathcal F\longrightarrow 0$$
and
\begin{equation}\label{fascio_F}
  0\lra \mathcal F \lra {\mathcal O}_S^{\oplus 3}(2M) \xrightarrow{(L_0,L_1,L_2)}  {\mathcal O}_S(4M)\lra0.
\end{equation}
%$$
%0 \longrightarrow {\mathcal F} \longrightarrow {\mathcal O}_S^{\oplus %3}(2M) \stackrel{(L_1, L_2, L_3)} \longrightarrow  {\mathcal %O}_S(4M) \longrightarrow 0
%$$

Using the first exact sequence and taking global sections, we obtain 
$$
0 \longrightarrow H^0(S,{\mathcal O}_S) ^{\oplus 3}\longrightarrow H^0(S,{\mathcal F}) \longrightarrow H^1({S,\mathcal O}_S(-2M))=0,
$$
where $H^1(S,{\mathcal O}_S(-2M)) \cong H^1(S,K_S(2M))^* =0$, by the Kodaira vanishing theorem. 
So $H^0(S,{\mathcal F})$ is $3$-dimensional. 
Moreover, we have 
$$
0 \longrightarrow H^1\left(S,{\mathcal O}_S^{\oplus 3}\right) \longrightarrow H^1(S,{\mathcal F}) \longrightarrow H^2(S,{\mathcal O}_S(-2M))\longrightarrow 0,
$$
so, since $H^2(S,{\mathcal O}_S(-2M)) \cong H^0(S,K_S(2M))^* \cong H^0(S,M)^* \cong {\mathbb C}^5$, we get $ H^1(S,{\mathcal F}) \cong {\mathbb C}^5 $. 

From now on we will denote by $Q'$ the proper transform of the curve $Q$ in $S$. 

Restricting the above exact sequence to $Q'$ and taking global sections, we obtain the following diagram: 
\begin{equation}\label{diagrammone}
\xymatrix@R=0.8cm@C=0.8cm{
 &  & H^0(S,\mathcal O_S(4M-Q')) \ar@{^{(}->}[d] \ar@{^{(}->}[r] & H^1(S,\mathcal F(-Q')) \ar[d]  \\
 H^0(S,\mathcal F)  \ar@{^{(}->}[d] \ar@{^{(}->}[r] & 
H^0(S,{\mathcal O}_S(2M))^{\oplus 3} \ar@{^{(}->}[d] \ar[r]^{f} & H^0(S,\mathcal O_S(4M))\ar@{->>}[d] \ar@{->>}[r] & H^1(S,\mathcal F) \ar@{->>}[d]\\
H^0(Q',{\mathcal F}|_{Q'}) \ar@{^{(}->}[r] & H^0\left(Q',\omega_{Q'}^{\otimes 2}\right)^{\oplus 3} \ar[r]^{\overline f} & H^0\left(Q',\omega_{Q'}^{\otimes 4}\right) \ar@{->>}[r]& H^1\left(Q',{\mathcal F}|_{Q'}\right)\rlap{.}
  }      
  \end{equation}

\begin{prop}
\label{r4-r8}
We have the following isomorphisms: $H^1(S,\mathcal F)\cong R^4_F$ and $H^1(Q',{\mathcal F}|_{Q'})\cong R^8_{Q}$. 
\end{prop}

\begin{proof}
The image of the map $f\colon H^0(S,{\mathcal O}_S(2M))^{\oplus 3} \rightarrow H^0(S,\mathcal O_S(4M))$ consists of those elements that can be written as $\sum_{i=0}^2 S_i L_i$, where $S_i \in H^0(S,\mathcal O_S(2M))$. So, by construction, it is the image of the restriction of $J^4_F$ to the surface $S$. This gives a map $R^4_F \rightarrow H^1(S, {\mathcal F})$ as follows: 
\begin{equation}
\xymatrix@R=0.8cm@C=0.8cm{
0\ar[r]& J^4_F\ar@{->>}[d]\ar[r] & H^0({\mathbb P}^4, \mathcal O_{\mathbb P^4}(4)) \ar@{->>}[d] \ar[r] & \ar[d] R^4_F \ar[r] & 0 \\
0 \ar[r]&\Im(f) \ar[r] & H^0(S,\mathcal O_S(4M)) \ar[r] & H^1(S, \mathcal F) \ar[r] & 0\rlap{.} 
    }      
\end{equation}

Thus the map $R^4_F \rightarrow H^1(S, {\mathcal F})$ is surjective, and since $\dim H^1(S, \mathcal F) =5 = \dim R^4_F$, it is an isomorphism. 

The isomorphism $H^1(Q',{\mathcal F}|_{Q'})\cong R^8_{Q}$ follows by observing that the map
$$
H^0(Q',\omega_{Q'}^{\otimes 2})^{\oplus 3} \cong H^0({\mathbb P}^2, {\mathcal O}_{{\mathbb P}^2}(4))^{\oplus 3} \longrightarrow H^0(Q',\omega_{Q'}^{\otimes 4}) \cong H^0(Q, {\mathcal O}_Q(8))
$$ is the map
$
(A_0, A_1, A_2) \longmapsto \sum_{i=0}^2 (A_i Q_{x_i})|_{Q}
$;
hence its image is $J^8_Q/S^3 \cdot Q$. 
\end{proof}

\begin{prop}
\label{sollevamento}
The image of the map $f\colon H^0(S,{\mathcal O}_S(2M))^{\oplus 3}\rightarrow H^0(S,\mathcal O_S(4M))$ contains the image of $H^0(S,\mathcal O_S(4M-Q')) $ in $H^0(S,\mathcal O_S(4M))$. 
%In particular the map $H^0(S, {\mathcal O}_S(4M)) \rightarrow R^4_F$ factors through  a map $ H^0(Q', \omega_{Q'}^{\otimes 4}) \rightarrow R^4_F$. 
\end{prop}
\begin{proof}

  First notice that the image of $H^0(S, \mathcal O_S(4M-Q')) =  H^0(S, \mathcal O_S(7H-3E))$ in $ H^0(S,\mathcal O_S(4M))$ has dimension $6$ and  is isomorphic to $H^0(S, \mathcal O_S(5H-2E))$ via the multiplication by $C' Q'$, where $C' $ is the strict transform of $C$ in $S$. By  \eqref{euler},  $C' Q'$ is in the image of the map
\[  
  H^0(S,{\mathcal O}_S(H))^{\oplus 3} \xrightarrow{(L_0,L_1,L_2)} H^0(S,{\mathcal O}_S(7H-2E)). 
\]
%$H^0(S,{\mathcal O}_S(H))^{\oplus 3} \stackrel{(L_1, L_2, L_3)}\longrightarrow H^0(S,\mathcal O_S(7H-2E))$.  

So the subspace of  $H^0(S,\mathcal O_S(4M))$ given by $C' Q' H^0(S, \mathcal O_S(5H-2E)) $ is in the image of the map $f\colon H^0(S,{\mathcal O}_S(2M))^{\oplus 3} \rightarrow H^0(S,\mathcal O_S(4M))$. 
\end{proof}

\begin{proof}[Proof of Proposition~\ref{lift}]
Set 
\[
\begin{aligned}
&\tilde{J}^4_F := \Im(f)\cong H^0\left(S,{\mathcal O}_S(2M)\right)^{\oplus 3}/H^0(S,{\mathcal F})  \,\text{ and }\,\\
& \tilde{J}^8_Q :=\Im(\overline f) \cong H^0\left(Q,\omega_Q^{\otimes 2}\right)^{\oplus 3} / H^0\left(Q,{\mathcal F}|_{Q}\right) \cong J^8_Q/H^0\left(Q,{\mathcal I}_Q(8)\right).
\end{aligned}
\]
%Notice that in the case of $D$ reduced we have that $S\subset \mathbb P^4$ and we also have that:
%$$\tilde{J}^4_F \cong J^4_F/H^0(S,{\mathcal I}_S(4)). $$

By Propositions~\ref{sollevamento} and~\ref{r4-r8},  diagram \eqref{diagrammone} yields the following diagram: 
\begin{equation}
\label{diagrammone2}
\xymatrix@R=0.8cm@C=0.8cm{
 \tilde{J}^4_F/H^0(S,\mathcal O_S(4M-Q'))  \ar@{^{(}->}[d] \ar@{^{(}->}[r] & 
  H^0(S,\mathcal O_S(4M))/ H^0(S,\mathcal O_S(4M-Q'))\ar[d]^{\cong} \ar@{->>}[r] & R^4_F \ar@{->>}[d]\\
 \tilde{J}^8_Q\ar@{^{(}->}[r] &  H^0\left(Q',\omega_{Q'}^{\otimes 4}\right) \ar@{->>}[r]&  R^8_{Q}.
  }      
  \end{equation}

 Hence we have a map  $H^0(Q',\omega_{Q'}^{\otimes 4}) \cong H^0(S,\mathcal O_S(4M))/ H^0(S,\mathcal O_S(4M-Q')) \rightarrow R^4_F$ lifting the map $\tau$. 
\end{proof}

\section{The second fundamental form in the locus of cubic threefolds}\label{sec6}

We would like to study some properties of the second fundamental form of the moduli space of cubic threefolds embedded in ${\mathcal A}_5$ via the period map. 

We have the cotangent exact sequence
\[0 \longrightarrow {N_{\cC\slash \cA_5}^*} \longrightarrow {\Omega^1_{\cA_5}}\vert_{\cC} \stackrel{q}\longrightarrow \Omega^1_{\cC} \longrightarrow 0.
\]

We consider the (orbifold) metric on ${\cA_5}$ induced by the Siegel metric and the corresponding Chern connection $\nabla$, and the second fundamental form 
\[
II_{\cC \slash \cA_5}\colon   {N_{\cC\slash \cA_5}^*} \longrightarrow \Sym^2 \Omega^1_{\cC},
\]
\[
II_{\cC \slash \cA_5} = (q \otimes Id) \circ \nabla|_{ {N_{\cC\slash \cA_5}^*}}.
\]

\begin{rem}
\label{totgeo}

It is well known that ${\mathcal C}$ is not totally geodesic in ${\mathcal A}_5$ (equivalently, the second fundamental form $II_{\cC \slash \cA_5}$ is not zero). In fact, the monodromy group of the family of cubic threefolds tensored with ${\mathbb Q}$ is the whole symplectic group (see \cite[Theorem~4]{beauville_mono}, see also \cite{huybrechts_notes}). 

%Another way to see this is observing that the Mumford-Tate group for the general cubic threefold is the whole symplectic group. Hence the closure $\overline{\mathcal C}$ of ${\mathcal C}$ in ${\mathcal A}_5$ is not a special subvariety. Since ${\mathcal C}$ contains $CM$-points, by a result of Moonen \cite{moonen}, $\overline{\mathcal C}$ is not totally geodesic (see Remark \cite{totgeo}). 
There are totally geodesic (indeed special) subvarieties of  ${\mathcal A}_5$ contained in $\overline{\mathcal C}$ given by the intermediate Jacobians of those cubic threefolds that are cyclic triple covers of ${\mathbb P}^3$ ramified on a cubic surface (\cite{act}, see also \cite{mt}). In this way Allcock, Carlson and Toledo proved that one can embed the moduli space of cubic surfaces in $\overline{\mathcal C}$ and showed that its closure is a ball quotient. 

A similar picture allowed them 
 to define a different period map for cubic threefolds which takes values in a ball quotient of dimension $10$ (see \cite{act1}). This is done using cubic fourfolds that are triple cyclic covers of ${\mathbb P}^4$ ramified over cubic threefolds and the Torelli theorem proved by Voisin in \cite{vo_invent}.

\end{rem}

To study $II_{\cC \slash \cA_5}$, we take its  pullback via the Prym map. 

Consider again diagram \eqref{diagrama1}. We set ${\mathcal H} :=  \Ker( {\Omega^1_{\cR_6}}|_{\mathcal {RQ}^-}    \rightarrow {T_F}^*) $ and $\pi:= (dP)^*\colon  P^*{\Omega^1_{\cA_5}}\vert_{\cC}  \rightarrow {\mathcal H}$. Finally, notice that $P^* \Omega^1_{\cC}  \cong \Ker(  \Omega^1_{\mathcal {RQ}^-} \rightarrow    {T_F}^*     )$. So we have the following diagram:
\begin{equation}\label{diagrama3}
\xymatrix@R=0.8cm@C=0.8cm{
              &  0 \ar[d] & 0 \ar[d] && \\
    &  \mathcal V^* \ar@{=}[r] \ar[d] & \mathcal{V}^* \ar[d] & 0\ar[d]& \\
      0 \ar[r] &        P^*{N_{\cC\slash \cA_5}^*} \ar[r] \ar[d] & P^*{\Omega^1_{\cA_5}}\vert_{\cC}  \ar[d]^{\pi} \ar[r]^{q_1} & P^* \Omega^1_{\cC} \ar[r] \ar@{=}[d]& 0
    \\
   0 \ar[r] & N_{\mathcal {RQ}^-}^*\vert_{\cR_6} \ar[r] \ar[d]  &{\mathcal H} \ar[d] \ar[r]^{q_2} &  P^* \Omega^1_{\cC}  \ar[r] \ar[d]& 0 
        \\
    &0&0&0.&&
    }      
\end{equation}
We consider the metric on $ P^*{\Omega^1_{\cA_5}}\vert_{\cC}$ induced by the Siegel metric which is degenerate on the tangent directions to the Fano surface, and the corresponding Chern connection that we still denote by $\nabla$. So we have an induced  metric on all the vector bundles of diagram \eqref {diagrama3} and the corresponding Chern connections. 

The pullback to $\mathcal {RQ}^-$ of the second fundamental form of the embedding of ${\cC}$ in ${\cA_5}$ is the second fundamental form associated with the first horizontal exact sequence: 
$$
II = (q_1 \otimes \Id) \circ \nabla\colon  P^*{N_{\cC\slash \cA_5}^*} \longrightarrow P^* \Omega^1_{\cC} \otimes  \Omega^1_{\mathcal {RQ}^-}. 
$$
\begin{rem}
\label{rem1}
The second fundamental form $II$ is the composition of the pullback $P^*II_{\cC \slash \cA_5}\colon P^*{N_{\cC\slash \cA_5}^*} \rightarrow \Sym^2 P^*\Omega^1_{\cC}$ with the map $\Sym^2 P^*\Omega^1_{\cC} \rightarrow  P^* \Omega^1_{\cC} \otimes  \Omega^1_{\mathcal {RQ}^-}$ induced by the inclusion $P^* \Omega^1_{\cC} \subset \Omega^1_{\mathcal {RQ}^-}$.
\end{rem}

We consider the second fundamental form  of the second vertical exact sequence: 
$$\widetilde{II} = (\pi \otimes \Id) \circ \nabla\colon {\mathcal V}^* \longrightarrow {\mathcal H} \otimes  \Omega^1_{\mathcal {RQ}^-} \subset {\Omega^1_{\cR_6}}|_{\mathcal {RQ}^-}  \otimes  \Omega^1_{\mathcal {RQ}^-} .$$

Notice that 
\begin{equation}\label{2ff}
II|_{{\mathcal V}^*} = (q_2 \otimes \Id) \circ \widetilde{II}\colon  {\mathcal V}^* \longrightarrow P^* \Omega^1_{\cC}  \otimes  \Omega^1_{\mathcal {RQ}^-}, 
\end{equation}
and we have 
\begin{equation}
\label{vstar}
    ({{\mathcal V}^*})_{Q_l} = I_2(\omega_{Q_l}\otimes \alpha).
\end{equation}

We will now recall the definition of Hodge-Gaussian maps introduced in  \cite[Theorem~4.4]{cpt} in the case of a smooth curve $Q$ together with a degree zero line bundle $\alpha$. 

Given a holomorphic line bundle $\alpha$ of degree zero on a curve $Q$, there exist a unique (up to a constant) hermitian metric $H$ on $\alpha$ and a
unique connection $D_H$ on $\alpha$, compatible with the holomorphic structure and the metric, which is flat. So we have $D_H = D'_H + \overline{\partial}$, where $D'_H$ is the $(1,0)$ component. 

Given a line bundle $L$ on $Q$, the Hodge-Gaussian map is a map 
\begin{align*}
\rho\colon I_2(L) \lra & H^0(Q, L \otimes \omega_Q \otimes \alpha^{-1}) \otimes H^0(Q, L \otimes \omega_Q \otimes \alpha)\\
&\cong \Hom(H^1(Q, L^{-1}  \otimes \alpha), H^0(Q, L \otimes \omega_Q \otimes \alpha)),
\end{align*}
defined as follows: 

We have the Dolbeault isomorphism $H^1(Q, L^{-1}  \otimes \alpha) \cong H^1_{\overline{\partial}}(Q, L^{-1}  \otimes \alpha) $. For any section $ s \in H^0(Q, L)$ and for any element $v \in H^1_{\overline{\partial}}(Q, L^{-1}  \otimes \alpha) $, take a Dolbeault representative $\theta$ of $v$, and consider the contraction $\theta s \in {\mathcal A}^1(\alpha)$. This is a $\overline{\partial}$-closed form with value in $\alpha$. So we have a decomposition 
$$\theta s = \gamma + \overline{\partial}h,$$
where $\gamma$ is the harmonic representative and $h$ is a ${\mathcal C}^{\infty} $ section of $\alpha$.

Take a basis $\{s_1,\ldots, s_n\}$ of $H^0(L)$, so that $\theta s_i = \gamma_i + \overline{\partial}h_i$ as above. 
A quadric $ \Gamma \in I_2(L)$ can be written as $\Gamma = \sum_{i,j=1}^n a_{ij} s_i \otimes s_j$, where $a_{i,j} = a_{j,i}$. If we choose local coordinates such that $s_i = f_i(z) l$, where $l$ is a local section of $L$, we have $\sum_{i,j=1}^n a_{ij} f_i(z) f_j(z) l^2\equiv 0$.

In \cite{cpt} it is shown that $\sum_{i,j=1}^n a_{ij} s_i D'_Hh_j$ is $\overline{\partial}$-closed, and $\rho$ is defined as follows:
\begin{equation}
\label{h-g}
\rho(\Gamma)(v) = \sum_{i,j=1}^n a_{ij} s_i D'_Hh_j \in H^0(Q, L \otimes \omega_Q \otimes \alpha).
\end{equation}
In  \cite[Theorem~4.5]{cpt} it is proven that the composition
\begin{equation}
\label{rho-mu2}
I_2(L) \stackrel{\rho}\lra H^0(Q, L \otimes \omega_Q \otimes \alpha^{-1}) \otimes H^0(Q, L \otimes \omega_Q \otimes \alpha) \stackrel{m} \lra H^0\left(Q,L^{\otimes 2} \otimes \omega_Q^{\otimes 2}\right) 
\end{equation}
equals the second Gaussian map  $\mu_{2, L}$ defined in \eqref{mu2L}, up to a constant. 

In \cite{cf-advances} this construction is used in the case where $Q$ is a smooth curve endowed with the $2$-torsion line bundle $\alpha$ and $L = \omega_Q \otimes \alpha $. Hence 
$$
\rho\colon I_2(\omega_Q \otimes \alpha) \lra H^0\left(Q,  \omega_Q^{\otimes 2}\right) \otimes H^0\left(Q,  \omega_Q^{\otimes 2}\right).
$$

If $g \geq 7$, the Prym map ${\mathcal R}_g \rightarrow {\mathcal A}_{g-1}$ is generically an immersion. At a point $(Q, \alpha)$ where it is an immersion, the conormal bundle  $N_{{\mathcal R}_g /\cA_{g-1}}^*$ is identified with $I_2(\omega_Q \otimes \alpha)$. 

In the proof of \cite[Theorem 2.1]{cf-advances}, it is shown that under these assumptions, the second fundamental form of the Prym map at $(Q, \alpha)$ equals the Hodge-Gaussian map $\rho$ (up to a constant).

For $g =6$, the Prym map is generically finite, so $I_2(\omega_Q \otimes \alpha) =0$ generically.

Nevertheless, we can use this construction in the case where $Q$ is a smooth plane quintic endowed with the $2$-torsion line bundle $\alpha$ since in this case, $I_2(\omega_Q \otimes \alpha)$ has dimension $2$ (see Proposition~\ref{basic_properties}).  So we have 
$$
\rho\colon I_2(\omega_Q \otimes \alpha) \lra H^0\left(Q,  \omega_Q^{\otimes 2}\right) \otimes H^0\left(Q,  \omega_Q^{\otimes 2}\right)
$$
and, up to a constant, $\mu_2 = m \circ \rho$. 

Since $ ({{\mathcal V}^*})_{Q} = I_2(\omega_{Q}\otimes \alpha)$ (see \eqref{vstar}),  varying the pair $(Q, \alpha) \in {\mathcal {RQ}^-}$, we can see $\rho$ as a map 
$$\rho\colon {\mathcal V}^* \longrightarrow  {\Omega^1_{\cR_6}}|_{\mathcal {RQ}^-}  \otimes   {\Omega^1_{\cR_6}}|_{\mathcal {RQ}^-}.$$

Denote by $t\colon H^0(Q,\omega_Q^{\otimes 2}) \cong H^0(Q,{\mathcal O}_Q(4)) = S^4 \rightarrow R^4_Q $ the natural projection and  by $g\colon R^4_Q  \otimes H^0(Q,\omega_Q^{\otimes 2})$ $\rightarrow \Sym^2 R^4_Q$ the projection induced by $t$.

\begin{thm}\label{rho}
The map  
$\rho\colon  {\mathcal V}^* \rightarrow  {\Omega^1_{\cR_6}}|_{\mathcal {RQ}^-}  \otimes   {\Omega^1_{\cR_6}}|_{\mathcal {RQ}^-}$
is a lifting of\, $\widetilde{II}$. At a point $(Q, \alpha) \in {\mathcal {RQ}^-}$, we have the commutative diagram
$$
\xymatrix@R=0.8cm@C=0.8cm{
                         I_2(\omega_{Q}\otimes \alpha) \ar[r]^{\rho \quad \quad \quad } \ar@{=} [d]    & H^0\left(Q,\omega_Q^{\otimes 2}\right) \otimes H^0\left(Q,\omega_Q^{\otimes 2}\right)  \ar[d]^{t \otimes \Id} 
    \\I_2(\omega_{Q}\otimes \alpha)\ar[r]^{\widetilde{II}}  \ar[d]^{(II_{\cC \slash \cA_5})_{\vert I_2(\omega_{Q}\otimes \alpha) }}  &R^4_Q  \otimes H^0\left(Q,\omega_Q^{\otimes 2}\right) \ar[d]^{g} \\
    \Sym^2R^2_F \ar@{^{(}->}[r]^{\iota}  &\Sym^2 R^4_{Q}.  
}
$$
 
\end{thm}

\begin{proof}
Notice that for all $Q$, we have $\widetilde{II} \colon  I_2(\omega_{Q}\otimes \alpha) \rightarrow \Hom({T_{\mathcal {RQ}^-,Q}}, H^0(Q, \omega_Q^{\otimes 2}))$. 
Then the proof of \cite[Theorem 2.1]{cf-advances} applies verbatim; hence for any $ v \in {T_{\mathcal {RQ}^-,Q}} \subset H^1(Q,T_Q)$ and for any $ \Gamma \in I_2( \omega_Q \otimes \alpha)$, we have (up to a constant)
$\rho(\Gamma) (v) =  \widetilde{II}(\Gamma)(v) \in H^0(Q,\omega_Q^{\otimes 2})$.  

For the reader's convenience, we repeat the argument.  

First of all, we take $v\in {T_{\mathcal {RQ}^-,Q}} \subset H^1(Q,T_Q)$, and we compute $II(\Gamma)(v)$ for every $\Gamma\in I_2(\omega_Q \otimes \alpha)$. Using the Kodaira Spencer map $k$, we can assume  $v=k\left(\frac{\partial}{\partial t}\right)$, where $t$ is the local coordinate of the unit disc $\Delta= \{|t|<1\}$ parametrizing a $1$-dimensional deformation $p\colon {\mathcal X} \rightarrow \Delta$, where $(Q,\alpha) = p^{-1}(0)$. Take a ${\mathcal C}^{\infty}$ lifting $Z$ of the holomorphic vector field $\frac{\partial}{\partial t}$ on $\Delta$, so that we have a ${\mathcal C}^{\infty}$ trivialization $\sigma\colon \Delta \times (Q,\alpha) \rightarrow {\mathcal X}$, $\sigma(t,x) := \Phi_{tZ}(1)$, where $\Phi_Z(t)$ is the flow of the vector field $Z$. Then $\theta:= \overline{\partial} Z|_{(Q,\alpha)}$ is a closed form in ${\mathcal A}^{0,1}(T_Q)$ such that $ [\theta]=v\in H^1(Q,T_Q)$. Denote by $(Q_t, \alpha_t)$ the fiber of $p$ over $t$, where $\alpha_t$ is a holomorphic line bundle in $\Pic^0(Q_t)[2]$ endowed with the flat structure induced by the double covering $\pi_t\colon \tilde{Q_t} \rightarrow Q_t$. We denote by $H_t$ the flat hermitian metric and by ${D_{H_t}} = D'_{H_t}+ \overline{\partial}_t$ the flat Chern connection.

Take a local section $\Omega$ of the bundle $p_*(\omega_{{\mathcal X}/\Delta} \otimes {\mathcal P})$, where $\mathcal P$ is the Prym bundle: ${\mathcal P}_t = \alpha_t$ on $Q_t$. For all $t \in \Delta$, $\omega(t) \in H^0(\omega_{Q_t} \otimes \alpha_t) \cong H^{1,0}(\alpha_t)$.  The vector bundle $P^*{\Omega^1_{\cA_5}}\vert_{\cC}$ on $\Delta$ is identified with $\Sym^2 p_*(\omega_{{\mathcal X}/\Delta} \otimes {\mathcal P})$; hence a local section of ${\mathcal V}^*$ at $t$ is a quadric $\tilde{\Gamma}(t) = \sum_{i,j} a_{i,j}(t) (\omega_i(t) \otimes \omega_j(t))$ such that $a_{i,j}(t) = a_{j,i}(t)$ for all $ i,j$ and such that $\sum_{i,j} a_{i,j}(t) \omega_i(t)\omega_j(t) \equiv 0$.

We now describe the computation of $\widetilde{II}$ at $\widetilde{\Gamma}(0) = \Gamma = \sum_{i,j} a_{i,j} (\omega_i \otimes \omega_j)$ along the tangent direction $v$.

Denote by $\sigma_t\colon Q \rightarrow Q_t$  the diffeomorphisms induced by $\sigma$, so we have a  map induced by pullback $\sigma_t^{*}\colon {\mathcal A}^1(\alpha_t) \rightarrow {\mathcal A}^1(\alpha)$. Since $\omega(t) \in {\mathcal A}^{1,0}(\alpha_t)$ is $
D_{H_t}$-closed, then  $\sigma_t^{*}(\omega(t))$ is also $D_H$-closed because we have $\sigma_t^{*}(D_{H_t}) = D_H$.

So $\sigma_t^{*}(\omega(t))$  has a power series expansion at $t=0$
$$
\sigma_t^{*}(\omega(t))= \omega + (\beta + D_H h)t +o(t),
$$
where $\omega := \omega(0)$, $\beta \in {\mathcal A}^1(\alpha)$ is harmonic and
$h$ is a ${\mathcal C}^{\infty}$ section of $\alpha$ (by the harmonic
decomposition for $D_H$). So we have
$\nabla^{GM}_{\frac{\partial}{\partial t}}[\omega(t)]_{t=0} =
[\beta]$, $\theta \cdot \omega = \beta^{0,1} +
\overline{\partial}h$, thus $k\left(\frac{\partial}{\partial t}\right) \cdot
[\omega] = [ \beta^{0,1}]$, where $  \beta^{0,1}$ is the $(0,1)$
component of $\beta$. Here we denote by ${\nabla}^{GM}$ the Gauss--Manin connection. 

Then for all $i$, we have
$\nabla^{1,0}_{\frac{\partial}{\partial t}}[\omega_i(t)]_{t=0} =
[{\beta_i}^{1,0}]$, so 
$$
\widetilde{II}(\Gamma)(v) = m \left(\nabla_{\frac{\partial}{\partial t}} \tilde{\Gamma}|_{t=0}\right) =
\sum_{i,j} a'_{i,j}(0) \omega_i \omega_j + 2 \sum_{i,j} a_{i,j}
\beta^{1,0}_i \omega_j.$$ 
Since $\sum_{i,j} a_{i,j}(t) \omega_i(t)
\omega_j(t) \equiv 0$,  its derivative with respect to $t$ at
$t=0$ must be zero; \textit{i.e.}, $2\sum_{i,j} a_{i,j}(\beta_i + D_H h_i)
\omega_j + \sum_{i,j} a'_{i,j}(0) \omega_i \omega_j \equiv 0$. Thus, 
if we take the $(1,0)$ part, we have $2\sum_{i,j}
a_{i,j}(\beta^{1,0}_i + D'_H h_i) \omega_j + \sum_{i,j} a'_{i,j}
\omega_i \omega_j \equiv 0$ and therefore $\widetilde{II}(\Gamma)(v) = -2 \sum_{i,j} a_{i,j}
\omega_j D'_H h_i. $ Now we conclude by \eqref{h-g}. 

Finally, the bottom part of the diagram follows from \eqref{2ff} and Remark~\ref{rem1}. 

\end{proof}

\begin{rem}
\label{rho+}
Notice that since the map $\iota$ in the diagram of Theorem~\ref{rho} is injective, the Hodge-Gaussian map $\rho$ determines the restriction of the second fundamental form to $I_2(\omega_{Q}\otimes \alpha)$. 

Moreover, the $2$-dimensional spaces $I_2(\omega_{Q_{l}} \otimes \alpha) \subset J^2_F$ parametrize the polars of the points in $l$; hence when we vary $l$ in the Fano surface, these subspaces  generate $J^2_F$. 

Therefore, the maps $\rho$ completely determine the second fundamental form $II_{\cC \slash \cA_5}$. On the other hand, since $II_{\cC \slash \cA_5}$ is not zero, for  generic $l$, the composition of $\rho$ with the projection $H^0(Q,\omega_Q^{\otimes 2}) \otimes H^0(Q,\omega_Q^{\otimes 2}) \rightarrow \Sym^2 R^4_{Q}$ is not zero, and its image is contained in that of $\iota\colon\Sym^2R^2_F  \hookrightarrow \Sym^2 R^4_{Q}$. 
\end{rem}

The following result shows that the image of $\rho$ is contained in a smaller subspace. This will be crucial further on.

Consider the multiplication map
$
m_{\alpha }\colon \Sym^2H^0(Q,\omega_Q \otimes \alpha)\rightarrow H^0(Q,\omega _Q^{\otimes 2}),
$
and define 
\[
\beta_{\alpha } =m_{\alpha }\otimes m_{\alpha }\colon
\Sym^2H^0(Q,\omega_Q \otimes \alpha)\otimes \Sym^2H^0(Q,\omega_Q \otimes \alpha)\longrightarrow 
H^0\left(Q,\omega _Q^{\otimes 2}\right)\otimes H^0\left(Q,\omega _Q^{\otimes 2}\right).
\]
\begin{prop}\label{tau_primo}
The image of $\rho $ is contained in the image of $\beta _{\alpha}$. In particular, there exists a well-defined map $\tau'\colon \Im(\beta_{\alpha })\rightarrow \Sym^2 R^2_F$.
\end{prop}

\begin{proof}
In Theorem~\ref{rho} we have shown that $(t\otimes \Id)\circ \rho =\widetilde{II}$. Composing with the natural map $\pi_l\colon T^*_{\mathcal {RQ}^-, l}\rightarrow T_{F(X), l}^*$ (dual of the differential of the inclusion $F(X)\subset \mathcal {RQ}^-$), we have that $\pi_l \circ \widetilde {II}=0$. This is a direct consequence of the definition of $\widetilde {II}$ since there is no variation of the periods along $F(X)$. Hence we have $\pi_l\circ (t\otimes \Id)\circ \rho =0$, and the first statement follows easily using the symmetry of $\rho $.

By diagram \eqref{diagramma3}, the restriction of $f'$ to the image of $m_{\alpha} $ equals $f$ followed by the inclusion $R^2_F \subset R^4_{Q_l}$. Hence we have a well-defined map $\tau'\colon \Im(\beta_{\alpha}) \rightarrow \Sym^2 R^2_F$. 
\end{proof}

\begin{thm}\label{diagramone}
Under the assumption of Section~\ref{sec4}, the following diagram is commutative: 
  $$
\xymatrix@R=0.8cm@C=0.8cm{
 I_2(\omega_{Q}\otimes \alpha) \ar[r]^{ \qquad \rho } \ar@{=} [d] \ar@/^2pc/[rrr] ^{\mu_2} 
 &
 \Im(\beta_{\alpha}) \ar@{^{(}->}[r] \ar[dddr]_{\tau'} 
 &
 H^0\left(Q,\omega_Q^{\otimes 2}\right) \otimes H^0\left(Q,\omega_Q^{\otimes 2}\right)  \ar[d]^{t \otimes \Id} \ar[r]^{\qquad m} 
 & H^0(Q,\omega_Q^{\otimes 4}) \ar[dd]^{\tau } \ar@/^2pc/[ddd] ^{\tilde{\tau}}
  \\
  I_2(\omega_{Q}\otimes \alpha) 
  \ar@/_3pc/[ddrr]^{(II_{\cC \slash \cA_5})\vert_{I_2(\omega_{Q}\otimes \alpha) }}
  \ar[rr]^{\widetilde{II}\qquad}  
  & &
  R^4_Q  \otimes H^0\left(Q,\omega_Q^{\otimes 2}\right) \ar[d]^{f'} &&  
  \\
  && \Sym^2R^4_Q \ar[r]^{m_Q} & R^8_Q 
      \\
    &  &\Sym^2R^2_F \ar[r]^{m_F} \ar@{^{(}->}[u] &R^4_F \ar[u]^{h}}
$$
where the  maps $m, m_F, m_Q$ are the natural multiplication maps.
In particular, 
\begin{equation}\label{tautilde_mu2}
    m_F \circ \left(II_{\cC \slash \cA_5}\right)\vert_{I_2(\omega_{Q}\otimes \alpha) }= \widetilde {\tau }\circ \mu_2.
\end{equation}
\end{thm}

\begin{proof}
The commutativity of the left-hand-side of the diagram is a corollary of Proposition~\ref{tau_primo} and Theorem~\ref{rho}.
 
The equality $m \circ \rho = \mu_2$ is equation \eqref{rho-mu2}.  By Proposition~\ref{lift}, we know that $h \circ \tilde{\tau} = \tau$. %Notice that the assumptions of section 4 are only used  to define %$\tilde{\tau}$.  

It remains to show that $m_F \circ \tau'= \tilde{\tau} \circ m $. First, we need to express $\tau'$ in terms of the Koszul cohomology of the surface $S$. Consider the Koszul complex obtained by tensoring \eqref{koszul_F} by $\mathcal O_S(-2M)$:
\begin{equation}\label{koszul_G}
0\lra {\mathcal O}_S(-4M) \lra {\mathcal O}_S^{\oplus 3}(-2M) \lra  {\mathcal O}_S^{\oplus 3} \xrightarrow{(L_0,L_1,L_2)} {\mathcal O}_S(2M)\lra 0.
\end{equation}
Define the sheaf $\mathcal G$ as the kernel of the last map:
\begin{equation}
    0\longrightarrow \mathcal G \longrightarrow {\mathcal O}_S^{\oplus 3} \longrightarrow  \mathcal O_S(2M) \longrightarrow 0.
\end{equation}
Notice that $\mathcal F=\mathcal G\otimes \mathcal O_S(2M)$.
One easily computes $H^0(S, \mathcal G)=0$; hence we have the following exact sequence in cohomology: 
\[
0\longrightarrow H^0(S,\mathcal O_S)^{\oplus 3} \stackrel{\kappa}{\longrightarrow} H^0(S,\mathcal O_S(2M)) \xrightarrow{\delta_{\mathcal G}} H^1(S,\mathcal G)\longrightarrow 0.
\]

The image of the map $\kappa $ consists of those elements that can be written as $\sum_{i=0}^2 A_i L_i$, where $A_i \in H^0(S,\mathcal O_S)$. So, by construction, it is the image of the restriction of $J^2_F$ to the surface $S$. This gives a map $R^2_F \rightarrow H^1(S, {\mathcal G})$ as follows: 
\begin{equation}
\xymatrix@R=0.8cm@C=0.8cm{
& 0 \ar[d]& 0 \ar[d]&& \\
&I_2(\omega_Q\otimes \alpha) \ar[d] \ar@{=}[r] & I_2(\omega_Q\otimes \alpha) \ar[d] && \\ 
0\ar[r]& J^2_F\ar@{->>}[d]\ar[r] & H^0({\mathbb P}^4, \mathcal O_{\mathbb P^4}(2)) \ar@{->>}[d] \ar[r] & \ar[d] R^2_F \ar[r] & 0 \\
0 \ar[r]&\Im(\kappa) \ar[r] & H^0(S,\mathcal O_S(2M)) \ar[r]^-{\delta_{\mathcal G}} & H^1(S, \mathcal G) \ar[r] & 0. 
    }      
\end{equation}

So the map $R^2_F \rightarrow H^1(S, \mathcal G)$ is surjective, and since they have the same dimension, it is an isomorphism. 

We notice that the image of the natural inclusion of $H^0(S,\mathcal O_S(2M))$ in $H^0(Q', \omega_{Q'}^{\otimes 2})$ given by restricting to $Q'$ coincides with the image of $m_{\alpha}$. 

Hence the map $\tau'$ is given by the tensor product
$$
\tau'= \delta_{\mathcal G} \otimes \delta_{\mathcal G}\colon H^0(S,\mathcal O_S(2M)) \otimes H^0(S,\mathcal O_S(2M)) \longrightarrow   H^1(S, \mathcal G) \otimes  H^1(S, \mathcal G).
$$

Now we consider the evaluation map of $\mathcal O_S(2M)$ tensored by $\mathcal G$:
\[
H^0(S,\mathcal O_S(2M))\otimes \mathcal G \longrightarrow \mathcal O_S(2M)\otimes \mathcal G \cong \mathcal F. 
\]
Denote by $\delta_{\mathcal F}$ the boundary map in the exacts sequence \eqref{fascio_F}. Then we have the following commutative diagram:
\begin{equation}
\xymatrix@R=0.8cm@C=0.8cm{
H^0(S,\mathcal O_S(2M))\otimes H^0(S,\mathcal O_S(2M)) \ar[d]_{\Id\otimes \delta_{\mathcal G}}\ar[r] & H^0(S,\mathcal O_S(4M)) \ar[d]^{\delta_{\mathcal F}}
\\
H^0(S,\mathcal O_S(2M))\otimes H^1(S,\mathcal G) \ar[r]^{\qquad \qquad \cup}  & 
H^1(S,\mathcal F),
}
\end{equation}
where the top horizontal map is given by multiplication of sections, while the bottom one is the cup product. 

By the commutativity of the above diagram, we see that 
$$ x \cup \delta_{\mathcal G}(y) = \delta_{\mathcal F}(x \cdot y) = \delta_{\mathcal G}(x) \cup y$$
for all $x,y \in H^0(S,\mathcal O_S(2M))$. Thus $\kappa(z)  \cup \delta_{\mathcal G}(y) = \delta_{\mathcal G}(\kappa(z)) \cup y =0$. 
Hence the cup product induces a map
$$
H^1(S,\mathcal G) \otimes H^1(S,\mathcal G) \longrightarrow H^1(S,\mathcal F),
$$
which is identified with the multiplication map $m_F\colon R^2_F \otimes R^2_F \rightarrow R^4_F$. Since $\tilde{\tau}$ is the map
$$
H^0(S,O_S(4M))/H^0(S,O_S(4M- Q'))\rightarrow H^1(S, {\mathcal F})
$$
induced by  $\delta_{\mathcal F}$, we have shown that
\begin{equation*}\pushQED{\qed}
  m_F \circ \tau'= \tilde{\tau} \circ m .
  \qedhere \popQED
	\end{equation*}
\renewcommand{\qed}{}   
\end{proof}

\begin{thm} The image of $m_F\circ (II_{\cC \slash \cA_5})\vert_{I_2(\omega_{Q}\otimes \alpha) }$ is contained in 
$\Ker(h\colon R^4_F\rightarrow R^8_Q)$; in particular, we have an endomorphism
\[
m_F\circ (II_{\cC \slash \cA_5})\vert_{I_2(\omega_{Q}\otimes \alpha) }\colon I_2(\omega_Q\otimes \alpha) \longrightarrow \Ker\left(h\colon R^4_F\longrightarrow R^8_Q\right) \cong I_2(\omega_Q\otimes \alpha).
\]
\end{thm}

\begin{proof}
Using the diagram of Theorem~\ref{diagramone}, we have that $\tau \circ \mu_2 = m_Q\circ \iota \circ (II_{\cC \slash \cA_5})\vert_{I_2(\omega_{Q}\otimes \alpha) }$.  By Proposition~\ref{mu2}, we have that $\tau \circ \mu_2=0$; hence the statement follows.
\end{proof}

\begin{thm} 
\label{identity}
The map $m_F\circ II_{\cC \slash \cA_5}$ is a multiple of the identity.
\end{thm}
\begin{proof}
Given a line $l$ contained in $X$, we have the commutative diagram
$$
\xymatrix@R=0.8cm@C=0.8cm{
0 \ar[d] && 0 \ar[d] \\
I_2(\omega_{Q_l}\otimes \alpha_l) \ar[d] \ar[rr] &&  I_2(\omega_{Q_l}\otimes \alpha_l)  \ar[d]
\\
    J^2_F \cong \left(R^1_F\right)^* \ar[d] \ar[r]^{II_{\cC \slash \cA_5}
} &  \Sym^2R^2_F   \ar[r]^{m_F} & R^4_F \cong \left(R^1_F\right)^* \ar[d]^{h}
    \\
   J^4_{Q}  \cong \left(R^1_{Q}\right)^* \ar[d] && R^8_{Q} \cong \left(R^1_{Q}\right)^* \ar[d]
   \\
    0 && 0\rlap{,}
    }      
$$
where the two vertical short exact sequences are both the dual of the exact sequence of Lemma~\ref{lemma_s}. So the corresponding $2$-dimensional subspace $I_2(\omega_{Q_l}\otimes \alpha_l)$  of $(R^1_F)^*$ parametrizing the polars of the points in $l$ is invariant under $m_F\circ II_{\cC \slash \cA_5}$.

Now consider  a  point $p\in X$, and choose two different lines $l_1, l_2$ contained in $X$ passing through $p$. Then the polar of $p$ belongs to the intersection of the two invariant $2$-dimensional subspaces $I_2(\omega_{Q_{l_i}} \otimes \alpha_{l_i})$; therefore, it is an eigenvector of the map $m_F\circ II_{\cC \slash \cA_5}$.  

All the quadrics in $J^2_F$ are polars of points in $\mathbb P^4$. Since all the polars of points in $X$ are eigenvectors, clearly all points in $J^2_F$ are eigenvectors. In fact, given a point $p \in \mathbb P^4$, consider a line $l$ passing through the point and intersecting $X$ in three points. The polars of the three points are eigenvectors, so the $2$-dimensional space $I_2(\omega_{Q_{l}} \otimes \alpha_l)$ contains three distinct eigenspaces. Hence the polars of all the points of the line $l$ are eigenvectors; in particular, the polar of $p$ is an eigenvector. So the map is a multiple of the identity.
\end{proof}

\section{Rank 3 quadrics}\label{sec7}

Recall that for a non-singular cubic threefold $X= \{F =0\}$, the polar quadrics of rank at most $4$ is the zero locus $H$ of the determinant of the Hessian matrix. 
In \cite[Lemma 37.4]{ar} it is proven that  the locus of polar quadrics of rank at most $3$ is equal to the singular locus of  $ H$. In \cite[Theorem 38.1]{ar} it is also proven that for a general cubic threefold $X$, the locus of rank at most $3$ polar quadrics is an irreducible non-singular curve intersecting $X$ at a finite number of points. In the next proposition we use the notation of Proposition~\ref{basic_properties}.

\begin{prop}
\label{rank3}
Assume that $X$ is a cubic threefold, and take a line $l$ contained in $X$ such that the  associated quintic $Q=Q_l$ is smooth. Assume that one of the points $\psi(p_i) \in X$ is such that the polar quadric $\Gamma_i := \Gamma_{\psi(p_i)}X$ is of rank $3$. Then the conic $C$ is the product of two lines  $l_1$ and $ l_2$, where  $l_1$ is  $4$-times tangent to $Q$ at $p_i$ and intersects $Q$ at another point $p_j$ which lies on the line $l_2$ that is bitangent to $Q$ at two other points. 
\end{prop}

\begin{proof}
Assume $i=1$ and, as usual, set $D = p_1 +p_2 +p_3 + p_4 +p_5$, where the $p_i$ are not necessarily distinct. As in Proposition~\ref{mu2}, set $L:= {\mathcal O}_Q(1)(-p_1)$, $N:= {\mathcal O}_Q(1)(p_1) \otimes \alpha $. The line bundle $N$ is isomorphic to ${\mathcal O}_Q(2)(-D +p_1)$.  
Since $\Gamma_1$ has rank $3$, we must have $|N|=|{\mathcal O}_Q(1)(p_1) \otimes \alpha| = |{\mathcal O}_Q(1)(-p_1)| + (R_1 + R_2)$, where $R_1+R_2$ is the base locus. In fact, assume $H^0(L) = \langle s_0, s_1\rangle$; then $H^0(N) = \langle ts_0, ts_1\rangle$, where $\div(t) = R_1 + R_2$ and $\Gamma_1 = s_0 (ts_0) \odot s_1 (ts_1) - s_0 (ts_1) \odot s_1 (ts_0) = t^2 (s_0^2 \odot s_1^2 - (s_0s_1)^2) \in I_2(\omega_Q \otimes \alpha)$.

Now we have to show that the points $R_i$ are in the support of $D$. 
 We have $\alpha = {\mathcal O}_Q(R_1 + R_2 -2p_1)$ and $4p_1 \sim 2R_1 + 2R_2$; hence $Q$ has a $g^1_4$. Since $Q$ is not hyperelliptic, $2p_1 \not\sim R_1 + R_2$, so there exists a $p_0 \in Q$ such that ${\mathcal O}_Q(1) = {\mathcal O}_Q(p_0 + 4p_1) = {\mathcal O}_Q(p_0 + 2R_1 + 2R_2)$. 

 Then ${\mathcal O}_Q(1)\otimes \alpha = {\mathcal O}_Q(p_1 + p_2 + p_3 +p_4 + p_5 ) = {\mathcal O}_Q(2p_1 + p_0 + R_1 +R_2)$, so $p_0 + p_1 + R_1 + R_2 \sim p_2 + p_3 + p_4 + p_5$. If $p_0 + p_1 + R_1 + R_2 \neq p_2 + p_3 + p_4 + p_5$, we get a $g^1_4$; hence $p_0, p_1, R_1, R_2$ would lie on a line, which is impossible since the line connecting $p_0$ and $p_1$ intersects $Q$ once in $p_0$ and with multiplicity $4$ in $p_1$. So we have $p_0 + p_1 + R_1 + R_2 =p_2 + p_3 + p_4 + p_5$, and up to a permutation of indexes, we can assume $p_0 = p_5$, $p_1 = p_4$, $R_1 = p_2$, $R_2 = p_3$, $D = 2p_1 + p_2 + p_3 + p_0$. We have ${\mathcal O}_Q(1) = {\mathcal O}_Q(p_0 + 4p_1) = {\mathcal O}_Q(p_0 + 2R_1 + 2R_2) = {\mathcal O}_Q(p_0 + 2p_2 + 2p_3)$; hence we conclude since the line $l_1$ is the line through $p_0$ and $p_1$, while the line $l_2$ through $p_0$ and $p_2$ is bitangent to $Q$ at $p_2$ and $p_3$ and $C = l_1 l_2$. 
\end{proof}

Denote by  ${\mathcal U}$ the family of smooth quintics such that there exist two lines $l_1$, $l_2$ intersecting at a point $p_0 \in Q$ and such that $l_1$ is $4$-times tangent to $Q$ at a point $p_1 \neq p_0$, while $l_2$ is bitangent to $Q$ at two other points $p_2$ and $p_3$. Then the quadric $\Gamma_1$ is of rank $3$.

Proposition~\ref{rank3} shows that all quadrics of rank $3$ which are given by polars to one of the points $\psi(p_i) $ of  on a cubic threefold $X$ arise in this way; namely, the associated quintic belongs to the family ${\mathcal U}$. In the next section we give an explicit description of ${\mathcal U}$, and a parameter count shows that this family provides a $10$-dimensional subvariety of $\mathcal M_6$. Since Proposition \ref{rank3} gives a distinguished $2$-torsion point, we get a well-defined $10$-dimensional variety $\mathcal {RU}\subset \mathcal {RQ}^-\subset \mathcal R_6$. Our aim is to prove that the restriction of the Prym map $P_{\mathcal U}\colon\mathcal {RU} \rightarrow \mathcal C$ is dominant. In order to have a proper map, we consider the closure $\overline{\mathcal {RU}}$ of $\mathcal {RU}$ in the Beauville partial compactification $\overline {\mathcal R}_6$. Now, it is enough to prove that one fiber of this extended map $\overline {P}_{\mathcal U}$ is finite. 
\begin{thm}
\label{dominant}
The map $\overline{P_{\mathcal U}}\colon \overline{\mathcal {RU}}\rightarrow \mathcal C$ is dominant. 
\end{thm}
\begin{proof}
In the proof of \cite[Theorem 38.1]{ar}, it is shown that on the Klein cubic threefold $X_0 = \{ F_0(x_0,\ldots,x_4) = x_0^2x_1+ x_4^2x_0+ x_1^2x_2 + x_2^2x_3 + x_3^2x_4=0\}$, there exist a finite number of rank $3$ polar quadrics of points in $X_0$. This then implies that on an open subset of $\mathcal C$ containing the Klein cubic, the locus of rank at most $3$ polar quadrics is an irreducible non-singular curve intersecting the cubic at a finite number of points. 
We will show that $X_0 $ is in the closure of the image of $\mathcal {RU}$. 

Notice that some rank $3$ polars of points of the Klein cubic are easily described, but they give rise to singular quintics. Hence we exhibit a deformation $X_{\epsilon}$ of $X_0$ containing one of these points with rank $3$ polar quadric and such that there exists a line $l$ through it  whose associated quintic $Q_l$ is smooth. By Proposition~\ref{rank3}, this implies that $X_0$ is in the closure of the image of $\mathcal {RU}$.

So consider the following deformation:
\begin{align*}
F_{\epsilon} (x_0,\ldots,x_4) &= K(x_0,\ldots,x_4) + \epsilon (x_1^2x_3 + x_3^2x_2)  \\
&=x_0^2x_1+ x_4^2x_0+ x_1^2x_2 + x_2^2x_3 + x_3^2x_4 + \epsilon (x_1^2x_3 + x_3^2x_2),
\end{align*}
 and set
$X_{\epsilon} = \{F_{\epsilon}(x_0,\ldots,x_4) = 0\}$. Let $p_0=[1:0:0:0:0] \in X_{\epsilon}$, and consider the rank $3$ quadric $\Gamma =  \Gamma_{p_0}X_{\epsilon} = \frac{\partial F_{\epsilon}}{\partial x_0} = 2x_0 x_1 + x_4^2$.  We want to show that there exists a line $l \subset X$ passing through $p_0$ such that the corresponding quintic $Q=Q_l$ is smooth, so that $l$ is as in  Proposition~\ref{rank3}. 

Consider the line $l = \{x_1 = x_2 = x_4 = 0\}$. Clearly, $l \subset X_{\epsilon}$ and $p_0 \in l$. A point in $l$ is $[u:0:0:v:0]$. Let $\Pi = \{x_0 = x_3 =0\}$, and take a point $p = [0:x:y:0:z]\in \Pi$. The fiber over $p$ of the conic bundle structure given by the projection from $l$ is obtained as the intersection of $t[0:x:y:0:z] + [u:0:0:v:0]$ with $X_{\epsilon}$ for $t \neq 0$. So, dividing by $t$, we get the conic 
$$\{[u:v:t] \ | \ xu^2+ z^2tu+ (z + \epsilon y)  v^2+ (y^2 + \epsilon x^2)  tv+  x^2y t^2= 0\} $$
with matrix
\[
\left(
\begin{array}{ccc}
    x & 0 & \frac{z^2}{2} \\[.5ex]
    0 & z + \epsilon y &  \frac{y^2 + \epsilon x^2}{2} \\[.5ex]
    \frac{z^2}{2} &  \frac{y^2 + \epsilon x^2}{2} & x^2y
\end{array}
\right).
\]

So the equation of the quintic $Q = Q_l$ is the determinant of the above matrix: 

$$Q = 4x^3yz + 2 \epsilon x^3y^2 - xy^4 - \epsilon^2x^5 - z^5 - \epsilon z^4y = 0.$$

One can check that for $\epsilon =1$, $Q$ is smooth. Moreover, the intersection of $Q$ with the line $\{x=0\}$  is $4q_1 + q_0$, where
$q_0 = [0:1:-\epsilon]$, $q_1 = [0:1:0]$. The intersection of $Q$ with the line $\{z = -\epsilon y\}$  is $q_0 + 2q_2 + 2q_3$, where $q_2 = [1: i \sqrt{\epsilon}: -i\epsilon\sqrt{\epsilon}]$, $q_3 = [-1: i \sqrt{\epsilon}: -i\epsilon\sqrt{\epsilon}]$. 

So $X_{\epsilon}$ belongs to the  image of $\mathcal {RU}$, and $X_0$ is in the closure of this image. 

%Using a projectivity sending $q_0$ to $[1:0:0]$, $q_2$ to  $[0:0:1]$, $q_3$ to $[1:0:1]$ and fixing $q_1$, the quintic %becomes the following element of the family ${\mathcal U}:$
%$$-16 x^4y-16(x^4z -2x^3z^2 + x^2z^3) -32(i-1) x^3yz +24(2i-1) x^2yz^2 + 8ixy^3z +8(1-3i)xyz^3 $$
%$$-4iy^3z^2+ +2 y^2z^3 +(4i-1)yz^4-zy^4.$$
\end{proof}

From the proof of Theorem~\ref{dominant}, we have the following. 
\begin{rem}
\label{adler+}
There exists a non-empty  Zariski open subset of the moduli space ${\mathcal C}$ of cubic threefolds $X$ such that $X$ contains a point whose polar is of rank $3$ and a line $l$ containing this point such that the quintic $Q_l$ is smooth. 
\end{rem}

\section{Main theorem}\label{sec8}

In this section we will prove the following.

\begin{thm}
\label{main}
The composition $m_F\circ II_{\cC \slash \cA_5}$ is identically zero. Equivalently, the image of the second fundamental form $II_{\cC \slash \cA_5}$ is contained in the kernel of the multiplication map. 
\end{thm}

The strategy of the proof is as follows: we will show that the restriction of $m_F\circ II_{\cC \slash \cA_5}$ to the image $\overline{{P}_{\mathcal U}}({\mathcal {RU}}) $ is zero. Since by Theorem~\ref{dominant}, the map $\overline{P_{\mathcal U}}\colon \overline{\mathcal {RU}}\rightarrow \mathcal C$ is dominant, this will imply that  $m_F\circ II_{\cC \slash \cA_5}$  is identically zero. 

We will compute the second Gaussian map $\mu_2(\Gamma_1)$ of the polar quadric $\Gamma_1$ of Proposition~\ref{rank3}.
We will then show that $0=(\tilde{\tau} \circ \mu_2)(\Gamma_1)=(m_F\circ II_{\cC \slash \cA_5})(\Gamma_1)$. Since by Theorem~\ref{identity}, $(m_F\circ II_{\cC \slash \cA_5})$ is a multiple of the identity, this shows that the restriction of $m_F\circ II_{\cC \slash \cA_5}$ to the image $\overline{{\mathcal P_{\mathcal U}}}(\mathcal {RU}) $ is zero. 

We start by giving an explicit description of the family ${\mathcal U}$. 

Fix the following points in ${\mathbb P}^2$: $p_0=[1:0:0]$, $p_1=[0:1:0]$, $p_2=[0:0:1]$, $p_3=[1:0:1]$. We want to consider all smooth quintics $Q$ passing through the points $p_i$ such that the line $l_1: = \overline{p_0p_1} = \{z=0\}$ is tangent at $Q$ in $p_1$ with multiplicity $4$ and the line $l_2 = \{y=0\}$ passing through $p_0$, $p_2$ and $p_3$ is bitangent at $Q$ in $p_2$ and in $p_3$. We also impose that $[1:1:1]$ belongs to the quintics.
One easily sees that the equation of such a quintic belongs to the following family:
\begin{align*}
\widetilde{\mathcal U} := \left\{Q(x,y,z) \right.= & a_1 x^4y + a_2\left(x^4z-2x^3z^2 + x^2z^3\right) + a_3 x^3yz + a_4 x^2y^2z + a_5 x^2yz^2 + a_6 xy^3z \\
&\left. +a_7xy^2z^2  +a_8xyz^3 +a_9 y^3z^2 + a_{10}y^2z^3 + a_{11}yz^4 + a_{12} y^4z = 0, \ \text{s.t.}  \sum_{i\neq 2}a_i=0\right\}. 
\end{align*}
Notice that $\widetilde{\mathcal U}$ has dimension $10$ and that  there is a finite number of projectivities acting on $\widetilde {\mathcal U}$. Therefore, 
the image $\mathcal U$ (of the smooth quintics in $\widetilde{\mathcal U}$) in $ \mathcal Q \subset \mathcal M_6$ is a $10$-dimensional irreducible subvariety.

Set $C = \{yz=0\}$. Then $C \cdot Q = 2p_0 + 4p_1 + 2p_2 +2p_3 = 2 D$, where $D = p_0 + 2p_1 +p_2 +p_3$.  We have $p_0 +2p_2 + 2p_3 \sim p_0 +4p_1$, hence $2p_2 + 2p_3 \sim 4p_1$. Thus, setting  
\begin{equation}\label{alpha}
\alpha: = {\mathcal O}_Q(p_2 + p_3 -2p_1),
\end{equation}
we have ${\alpha}^{\otimes 2} \cong  {\mathcal O}_Q$ and $\alpha \not\cong {\mathcal O}_Q$. 

As in Propositions~\ref{mu2} and~\ref{rank3}, set $L:= {\mathcal O}_Q(1)(-p_1) =   {\mathcal O}_Q(p_0 + 3p_1)$ and $N:= {\mathcal O}_Q(1)(p_1) \otimes \alpha = {\mathcal O}_Q(p_0 + 3p_1 +p_2 + p_3)$. The line bundle $N$ is isomorphic to ${\mathcal O}_Q(2)(-p_0 -p_1 -p_2-p_3)$, which is the linear system of conics passing through $p_0, p_1, p_2, p_3$. Therefore, they all contain the line $l_2$ as base locus,  and since $\div(l_2)(-p_0-p_2-p_3) = p_2 + p_3$, the base locus of $|N|$ is $p_2 + p_3$; hence $|N|  = (p_2 + p_3) + |p_0 + 3p_1|= (p_2 + p_3) + |L| $.  Assume $H^0(L) = \langle s_0, s_1\rangle$; then $H^0(N) = \langle ts_0, ts_1\rangle$, where $\div(t) = p_2 + p_3$. The quadric $\Gamma_1 := s_0 (ts_0) \odot s_1 (ts_1) - s_0 (ts_1) \odot s_1 (ts_0) = t^2 (s_0^2 \odot s_1^2 - (s_0s_1)^2) \in I_2(\omega_Q \otimes \alpha)$ has rank $3$, and we have 
$\rho_1 = \mu_2(\Gamma_1) = \mu_{1,L}(s_0 \wedge s_1) \cdot   \mu_{1,N}(ts_0 \wedge ts_1)$.
So
\begin{align*}
  \div\left(\rho_1\right) & = \div \left(\Gamma_{p_1}(Q)\right) -2p_1 + \div \left(\Gamma_{p_1}(Q)\right) -2p_1 + 2(p_2 + p_3) \\
  & = 2 \div (\Gamma_{p_1}(Q))-4 p_1 + 2(p_2+p_3) =  2 \div \left(\Gamma_{p_1}(Q)\right) + \div\left(\frac{y}{z}\right),
  \end{align*}
where, as usual, $\Gamma _x(Q)$ denotes the polar of $x$ with respect to $Q$. 
We have $\div(\Gamma_{p_1}(Q)) = 4p_1 + D_1$ for some degree $16$ effective divisor $D_1$, so $\div(\rho_1) = \div(\Gamma_{p_1}(Q)) + D_1 + 2p_2 +2 p_3$. Therefore, since $\div(\rho_1)  \in |{\mathcal O}_Q(8)| $ and $\div(\Gamma_{p_1}(Q)) \in |{\mathcal O}_Q(4)|$, we have $D_1 + 2p_2 +2 p_3 \in |{\mathcal O}_Q(4)|$. Since $H^0({\mathbb P}^2, {\mathcal O}_{{\mathbb P}^2}(4)) \cong H^0(Q, {\mathcal O}_{Q}(4))$, there exists an $h \in H^0({\mathcal O}_{{\mathbb P}^2}(4))$ such that 
$$
h \cdot \Gamma_{p_1}(Q) = \rho_1 =   \frac{y}{z} \cdot \left(\Gamma_{p_1}(Q)\right)^2  \ \ \   \  \mod \ Q,
$$
so $zh \Gamma_{p_1}(Q)= y (\Gamma_{p_1}(Q))^2   + RQ$ for some $R$; hence $\Gamma_{p_1}(Q) | R$, and there exists a constant $\lambda$ such that 
$$zh = y \Gamma_{p_1}(Q) + \lambda Q = y Q_y + \lambda Q.$$

A simple computation shows that taking $\lambda = -1$, we get
$$
y Q_y - Q= z\left(-a_2\left(x^4-2x^3z + x^2z^2\right)+ a_4 x^2y^2 + 2a_6 x y^3 + a_7 x y^2 z + 2 a_9 y^3 z + a_{10} y^2 z^2 + 3a_{12} y^4\right),
$$
hence $h= -a_2(x^4-2x^3z + x^2z^2)+ a_4 x^2y^2 + 2a_6 x y^3 + a_7 x y^2 z + 2 a_9 y^3 z + a_{10} y^2 z^2 + 3a_{12} y^4$. 

Denote by $|\overline{M}|$ the linear system of plane cubics passing through the points $p_0, p_1, p_2, p_3$ and whose tangent in $p_1$ is the line $l_1 = \{z=0\}$. 
A basis for the linear system $|\overline{M}|$ is  $\{x^2y, \ x^2z - xz^2, \  y^2z, \ xyz, \ yz^2\}$. A basis for $|2\overline{M}|$ is then easily computed as  
\begin{multline*}
\left\{x^4y^2, \ \left(x^4yz-x^3yz^2\right), \ x^2y^3z, \ x^3y^2z, \ x^2y^2z^2, \ \left(x^4z^2 + x^2z^4 -2x^3z^3\right), \ \left(x^3yz^2-x^2yz^3\right),\right.\\ 
\left.\left(x^2yz^3-xyz^4\right),  \ y^4z^2, \ xy^3z^2, \ y^3z^3, \ xy^2z^3, \  y^2z^4\right\}.
\end{multline*}
As in Remark~\ref{quartic-sextic}, we see that for any sextic $A \in H^0(\mathbb P^2,2\overline{M})$, there exist a quartic $\tilde{A}$ and a linear form $\gamma $ such that $A = C \tilde{A} + \gamma Q= yz \tilde{A} + \gamma Q$.  In fact, consider the restriction map
$$
H^0\left({\mathbb P}^2, 2\overline{M}\right) \longrightarrow H^0\left(Q, 2\overline{M}|_{Q}\right) = H^0\left(Q, \omega_Q^{\otimes 2}\right) = H^0\left(Q, {\mathcal O}_Q(4)\right) \cong H^0\left({\mathbb P}^2, {\mathcal O}_{{\mathbb P}^2}(4)\right),
$$ 
 which is injective.
Take a sextic $A \in H^0({\mathbb P}^2, 2\overline{M})$; then there exists a quartic $\tilde{A} \in H^0({\mathbb P}^2, {\mathcal O}_{{\mathbb P}^2}(4))$ such that  $\div(A|_{Q}) = \div(\tilde{A}|_{Q}) + 2D$. Hence $A|_{Q} = (\tilde{A} C)|_{Q}$, so  there exists a linear form $\gamma$ such that $A = \tilde{A} C + \gamma Q$.

We find a basis for the linear system $\Gamma$ of such quartics as follows. First we take all the elements of the given basis of $|2\overline{M}|$ which are divisible by $yz$, and we divide them by $yz$. In this way, we find $11$  quartics. Then we add $Q_z$ and $Q_y$ since $Q_y C - C_y Q = Q_y (yz) - zQ \in H^0(2\overline{M})$ and $Q_z C - C_z Q = Q_z (yz) - yQ  \in H^0(2\overline{M})$ by Lemma~\ref{Li-Ti}. 
So we get the following basis for $\Gamma$: 
\begin{multline*}
  \left\{\left(x^4-x^3z\right), \ x^2y^2, \ x^3y, \ x^2yz, \ \left(x^3z - x^2z^2\right), \ (x^2z^2-xz^3), \ y^3z, \ xy^2z, \ y^2z^2, \ xyz^2, \ yz^3, \right.\\
\left. \left(a_6xy^3 + a_{12}y^4\right), \ \left(a_1x^4 + a_3 x^3z + a_5 x^2z^2 + a_8xz^3 + a_{11} z^4\right)\right\}.
\end{multline*}

\begin{lem}
\label{Ai}
There exist $A_i \in |2\overline{M}|$, $i=0,1,2$,  and a cubic $S$ such that 
$$A_0 Q_x + A_ 1 Q_y + A_2 Q_z - (yz) hQ_y =  (yz)SQ.$$
\end{lem}
\begin{proof}
By the above discussion, the sextic $A_i$ corresponds to a quartic $\tilde{A_i} \in \Gamma$ and a linear form $\gamma_i$ such that $A_i = (yz)\tilde{A_i} + \gamma_i Q$. So the statement can be rewritten as 
$(yz)( \tilde{A_1} Q_y + \tilde{A_ 0} Q_x + \tilde{A_2} Q_z)- (yz) hQ_y =  (yz)RQ = (yz) \frac{1}{5} R (yQ_y + xQ_x + zQ_z)$  for some cubic $R$. Hence we look for quartics $\tilde{A_i} \in \Gamma$ and for a cubic $T$ such that  $Q_y(\tilde{A_1} -h - Ty) + Q_x(\tilde{A_0} -Tx) + Q_z ( \tilde{A_2} - Tz) =0$.
Using the above basis of $\Gamma$, one immediately sees that $h \equiv a_{12} y^4 \mod \ \Gamma$. So if we take $T = a_6 xy^2$, we have 
\begin{gather*}
Ty = a_6 x y^3 \equiv -h = -a_{12} y^4 \mod \ \Gamma,\\
Tx = a_6 x^2y^2 \in \Gamma,\\
Tz = a_6 xy^2z \in \Gamma.
\end{gather*}
So we set $\tilde{A_1} = h + Ty \in \Gamma$, $\tilde{A_0} = Tx \in \Gamma$, $\tilde{A_2} = Tz \in \Gamma$, and we conclude the proof. 
\end{proof}

From Lemmas~\ref{Li-Ti} and~\ref{Ai},  we immediately get the following. 

\begin{prop}
\label{beta1}
The element $\rho_1 = \mu_2(\Gamma_1) \in H^0(Q, \omega_Q^{\otimes 4}$) belongs to the image of the map 
\begin{align*}
  H^0({\mathbb P}^2, 2\overline{M})^{\oplus 3} &\longrightarrow H^0(Q, \omega_Q^{\otimes 4}),\\
  (A_0, A_1, A_2) &\longmapsto (A_0T_0+ A_1 T_1 + A_2 T_2)|_{Q} -4 D.
  \end{align*}
In particular, $\tilde {\tau } (\mu_2 (\Gamma_1))=0$ in $R^4_F$.
\end{prop}

\begin{proof}
Set $x_1 =y$, $x_0 = x$, $x_2 = z$. Recall that by Lemma~\ref{Li-Ti}, we have $(\sum A_i  T_i)|_{Q} = (\sum A_i Q_{x_i})|_{Q} + C \cdot Q = (\sum A_i Q_{x_i})|_{Q} + 2 D$. 

We have $\div((hQ_{x_1})|_{Q}) = \div(\rho_1)$. So by Lemma~\ref{Ai}, we get  $(\sum A_i Q_{x_i})|_{Q} = 2 D + \div((h Q_{x_1})|_{Q})=2D +  \div(\rho_1)$. Thus  we obtain
$$
\left(\sum A_i T_i\right)|_{Q} - 4 D = \left(\sum A_i Q_{x_i}\right)|_{Q}- 2 D = \div\left(\rho_1\right).
$$
By the definition of $\tilde{\tau}$, we see that $\tilde{\tau}(\mu_2(\Gamma_1)) =0$.
\end{proof}

\begin{cor}
\label{zero}
For the cubic threefold $X$ associated with a quintic in ${\mathcal U}$ together with the $2$-torsion element $\alpha$ as in \eqref{alpha}, the composition $m_F\circ II_{\cC \slash \cA_5}$ is identically zero. 
\end{cor}
\begin{proof}
By Theorem~\ref{identity},  $m_F\circ II_{\cC \slash \cA_5}$ is a multiple of the identity. By Theorem~\ref{tautilde_mu2}, we have 
$$ \tilde{\tau} \circ \mu_2 = m_F\circ (II_{\cC \slash \cA_5})|_{I_2(\omega_Q \otimes \alpha)}.$$
In Proposition~\ref{beta1}, we have shown that $(\tilde{\tau} \circ \mu_2)(\Gamma_1) =0$, so this concludes the proof. 

\end{proof}

To conclude the proof of Theorem~\ref{main}, notice that Corollary~\ref{zero} says that the restriction of $m_F\circ II_{\cC \slash \cA_5}$ to the image $\overline{{P}_{\mathcal U}}({\mathcal {RU}}) $ is zero. By Theorem~\ref{dominant}, the map $\overline{P_{\mathcal U}}\colon \overline{\mathcal {RU}}\rightarrow \mathcal C$ is dominant, so  $m_F\circ II_{\cC \slash \cA_5}$  is identically zero.

%%%%%%%%%%%%%%%%%%%%%
% References
%%%%%%%%%%%%%%%%%%%%%

\providecommand{\bysame}{\leavevmode\hbox to3em{\hrulefill}\thinspace}
\providecommand{\href}[2]{#2}

\end{document}